\documentclass{theoretics}
\title{Directed Acyclic Outerplanar Graphs Have Constant Stack Number}

\ThCSauthor[KIT]{Paul Jungeblut}{paul.jungeblut@kit.edu}[https://orcid.org/0000-0001-8241-2102]
\ThCSauthor[KIT]{Laura Merker}{laura.merker2@kit.edu}[https://orcid.org/0000-0003-1961-4531]
\ThCSauthor[KIT]{Torsten Ueckerdt}{torsten.ueckerdt@kit.edu}[https://orcid.org/0000-0002-0645-9715]
\ThCSaffil[KIT]{Karlsruhe Institute of Technology, Germany}
\ThCSshortnames{P. Jungeblut, L. Merker, T. Ueckerdt}

\ThCSshorttitle{Directed Acyclic Outerplanar Graphs Have Constant Stack Number}
\ThCSyear{2025}
\ThCSarticlenum{25}
\ThCSreceived{Apr 23, 2024}
\ThCSrevised{May 30, 2025}
\ThCSaccepted{Jul 26, 2025}
\ThCSpublished{Oct 17, 2025}
\ThCSdoicreatedtrue

\ThCSkeywords{Linear layouts, stack number, directed graphs, outerplanar graphs, 2-trees}

\ThCSthanks{A preliminary version of this paper is published in the proceedings of the 64th Annual Symposium on Foundations of Computer Science (FOCS)~\copyright~2023~IEEE~\cite{Jungeblut2023_FOCSversion}.}

\usepackage{mathtools}
\usepackage{xspace}
\usepackage{csquotes}
\usepackage{complexity}
\usepackage[capitalise,noabbrev,nameinlink]{cleveref}
\usepackage{subcaption}
\graphicspath{{figures/}}
\crefname{sidefigure}{Figure}{Figures}
\crefname{cfigure}{Figure}{Figures}
\crefname{csidefigure}{Figure}{Figures}
\crefname{conjecture}{Conjecture}{Conjectures}
\crefname{open}{Open Problem}{Open Problems}
\crefname{claim}{Claim}{Claims}

\DeclareMathOperator{\sn}{sn}
\DeclareMathOperator{\tn}{tn}

\newcommand{\calP}{\ensuremath{\mathcal{P}}\xspace}

\newcommand{\calB}{\ensuremath{\mathcal{B}}\xspace}

\newcommand{\invref}[1]{\labelcref{#1}\xspace}
\newcommand{\propref}[1]{\labelcref{#1}\xspace}

\newcommand{\lab}[1]{\texttt{#1}\xspace}
\newcommand{\Hpartition}{$H$\nobreakdash-partition\xspace}
\colorlet{lipicsGray}{@ThCSlightred}

\makeatletter
\newcommand\IfRestateTF{%
  \ifx\label\thmt@gobble@label 
    \expandafter\@firstoftwo
  \else
    \expandafter\@secondoftwo
  \fi
}
\makeatother
\newcommand{\RestateRemark}{\IfRestateTF{{\normalfont\bfseries (Restated) }}{}}

\begin{document}

\maketitle

\begin{abstract}
    The stack number of a directed acyclic graph~$G$ is the minimum~$k$ for which there is a topological ordering of~$G$ and a $k$-coloring of the edges such that no two edges of the same color cross, i.e., have alternating endpoints along the topological ordering.
    We prove that the stack number of directed acyclic outerplanar graphs
    is bounded by a constant, which gives a positive answer to a conjecture by Heath, Pemmaraju and Trenk~[SIAM J.\ Computing,~$1999$].
    As an immediate consequence, this shows that all upward outerplanar graphs have constant stack number, answering a question by Bhore et al.~[Eur.\,J.\,Comb.,~2023] and thereby making significant progress towards the problem for general upward planar graphs originating from Nowakowski and Parker [Order, 1989]. 
    As our main tool we develop the novel technique of directed $H$-partitions, which might be of independent interest.

    We complement the bounded stack number for directed acyclic outerplanar graphs by constructing a family of directed acyclic $2$-trees that have unbounded stack number, 
    thereby refuting a conjecture by Nöllenburg and Pupyrev [GD 2023].
\end{abstract}


\section{Introduction}

A directed acyclic graph, shortly DAG, is a directed graph with no directed cycles.
For an integer $k \geq 1$ and a directed acyclic graph~$G$, a \emph{$k$-stack layout} of $G$ consists of a topological ordering~$\prec$ of the vertices~$V(G)$ and a partition of the edges~$E(G)$ into~$k$ parts such that each part is a \emph{stack}.
A part is a stack if no two of its edges cross with respect to~$\prec$, where two edges~$ab$ and~$cd$ \emph{cross} if their endpoints are ordered $a \prec c \prec b \prec d$ or $c \prec a \prec d \prec b$.
Now the \emph{stack number}\footnote{
    The stack number is also known as \emph{book thickness} or \emph{page number}, especially in the older literature (stacks are called pages then).
    In recent years the term stack number seems to be preferred over the others as it explicitly expresses the first-in-last-out property of every part, in alignment with related concepts like queue layouts.
} $\sn(G)$ of~$G$ is the minimum~$k$ such that there exists a $k$-stack layout of~$G$.
It is convenient to interpret the ordering $\prec$ as the vertices of $G$ being laid out from left to right, i.e., for $a \prec b$ we say that $a$ is to the left of $b$ (and $b$ is to the right of $a$) or that $a$ comes before $b$ in $\prec$ (and $b$ comes after $a$ in $\prec$).
Then $\prec$ being a topological ordering of $G$ means that for every edge $ab$ directed from $a$ to $b$, vertex $a$ must come before vertex $b$ in $\prec$, in other words, every edge is directed from its left to its right endpoint according to $\prec$.

Given a topological ordering~$\prec$ of $G$, a partition of $E(G)$ into $k$ stacks can equivalently be seen as a $k$-edge coloring, such that each color class is crossing-free.
The simplest obstruction to admitting a partition into $k-1$ stacks is a set of $k$ pairwise crossing edges, also called a \emph{$k$-twist}.
For example, consider the $2k$-vertex graph $G_k$, $k \geq 2$, in the left of \cref{fig:planar-partial-3-tree} consisting of the directed path $P = (\ell_1,\ldots,\ell_k,r_1,\ldots,r_k)$ and the matching $M = \{\ell_i r_i \mid i=1,\ldots,k\}$.
This directed acyclic graph has only one topological ordering $\prec$ in which the vertices are ordered along the directed path $P$.
However with respect to this ordering $\prec$ the edges in $M$ form a $k$-twist.
It follows that $\sn(G_k) \geq k$ for all $k \geq 2$, which is also tight, as we can easily find a partition of $E(G_k)$ into $k$ stacks as indicated in the right of \cref{fig:planar-partial-3-tree}.

\begin{sidefigure}
    \centering
        \includegraphics[scale=1.30909090909, page=2]{figures/planar-partial-3-tree.pdf}
    \hspace{2em}
    \includegraphics[scale=1.30909090909, page=1]{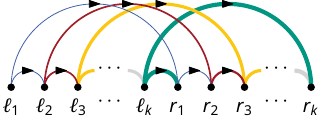}
    \caption{A planar directed acyclic graph $G_k$ of treewidth~$3$ on $2k$ vertices (left) and its unique topological ordering $\prec$ containing a $k$-twist (right).}
    \label{fig:planar-partial-3-tree}
\end{sidefigure}

Evidently, for every $k$ the graph $G_k$ is planar.
Moreover, $G_k$ is $2$-degenerate and has treewidth\footnote{We do not need the formal definition of treewidth here but we define related relevant concepts like $2$-trees in \cref{sec:preliminaries}.} at most~$3$.
So this family of DAGs shows that the stack number is not bounded within the class of all planar directed acyclic graphs; not even within those that are $2$\nobreakdash-degenerate and have treewidth at most~$3$ (even pathwidth at most~$3$).
On the other hand, as already noted by Nowakowski and Parker~\cite{Nowakowski1989_DirectedStackNumer} as well as Heath, Pemmaraju and Trenk~\cite{Heath1999_DAGs1}, it is easy to verify that if $G$ is a directed forest (equivalently, if $G$ is $1$-degenerate, or if $G$ has treewidth $1$), then $\sn(G) = 1$.
However, determining the largest stack number among the class of DAGs of treewidth $2$, in particular in the special case of outerplanar graphs, has remained an intriguing open problem for several decades~\cite{Heath1999_DAGs1,DiGiacomo2006_SeriesParallel,Noellenburg2021_DAGsWithConstantStackNumber,Bekos2019_Planar_kPlanar,Bhore2021_UpwardBookThickness}.
In fact there was little progress even for the considerably smaller class of outerplanar DAGs, for which a well-known conjecture of Heath, Pemmaraju and Trenk~\cite{Heath1999_DAGs1} from $1999$ states that the stack number should be bounded.

\begin{conjecture}[Heath, Pemmaraju and Trenk, 1999~\cite{Heath1999_DAGs1}]\label{conj:outerplanar}
 The stack number of the class of directed acyclic outerplanar graphs is bounded above by a constant.
\end{conjecture}

For the sake of an example, let us argue that the outerplanar DAG $G$ in \cref{fig:3-fence-example} has $\sn(G) \geq 3$.
In fact, due to symmetry it is enough to consider a topological ordering $\prec$ with $a_3$ to the left of $b_1$ and to observe that in $\prec$ the edges $a_1b_1$, $a_2b_2$, $a_3b_3$ form a $3$-twist.
\cref{conj:outerplanar} concerning outerplanar DAGs, but also the more general question about directed acyclic $2$-trees, have received increasing attention over the last years
(see \cref{sec:related_work} for more details), but it remained open to this day whether either of these classes contains DAGs of arbitrarily large stack number.

\begin{sidefigure}
    \centering
    \includegraphics[scale=1.30909090909, page=2]{figures/3-fence-example}
    \hspace{2em}
    \includegraphics[scale=1.30909090909, page=1]{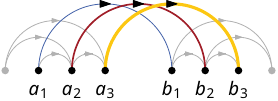}
    \caption{An outerplanar directed acyclic graph on eight vertices (left) and a topological ordering $\prec$ with $a_3$ to the left of $b_1$ containing a $3$-twist (right).}
    \label{fig:3-fence-example}
\end{sidefigure}

In addition to \cref{conj:outerplanar}, there is a second major open question in the field of directed linear layouts which we attack.
Here, a DAG is called \emph{upward planar} if it can be drawn in the plane such that the edges are crossing-free and $y$-monotone.

\begin{open}[Nowakowski, Parker, 1989~\cite{Nowakowski1989_DirectedStackNumer}]\label{prob:upward_sn_bounded}
    Is the stack number of the class of upward planar graphs bounded above by a constant?
\end{open}

\subparagraph{Our results.}
In this paper, we answer both long-standing open problems connected to \cref{conj:outerplanar}, that is whether or not outerplanar graphs and 2-trees have bounded stack number. 
We answer the first question in the positive, the second in the negative.
Additionally, we attack \cref{prob:upward_sn_bounded} and contribute a significant class of upward planar graphs with bounded stack number.

In \cref{sec:outerplanar}, we prove that \cref{conj:outerplanar} is true by showing that outerplanar DAGs have stack number at most $24776$.
As one of our main tools we introduce \emph{directed \Hpartition{}s}, which may be of independent interest for investigations of directed graphs in general.
We remark that while another variant of $ H $-partitions, so called \emph{layered $ H $-partitions}, caused a breakthrough in the investigation of queue layouts (a notion closely related to stack layouts)~\cite{Dujmovic2020_PlanarGraphsBoundedQueueNumber}, this is -- to the best of our knowledge -- the first time that $ H $-partitions are successfully adjusted to work with stack layouts.

\begin{restatable}{theorem}{outerplanarBounded}
    \label{thm:outerplanar_bounded}\RestateRemark
    The stack number of outerplanar DAGs is bounded by a constant.
    Moreover, every outerplanar DAG~$G$ has $\sn(G) \leq 24776$.
\end{restatable}

Our second main result complements the constant upper bound for outerplanar DAGs by showing that already a just slightly larger subclass of graphs of treewidth~$2$ has unbounded stack number.%
\footnote{%
    We remark that every 2-tree can be constructed by iteratively \emph{stacking} vertices onto an edge, i.e., by introducing a new vertex that is connected exactly to the endpoints of an edge.
    A directed 2-tree is called \emph{monotone} if in each step the new vertex is a source or a sink.
    See also \cref{sec:preliminaries}.
}

\begin{restatable}{theorem}{twotreesUnbounded}
    \label{thm:2tree_unbounded}\RestateRemark
    The stack number of DAGs of treewidth~$2$ is unbounded.
    Moreover, for every~$k \geq 1$ there exists a monotone $2$-tree~$G$ with~$\sn(G) \geq k$ in which at most two vertices are stacked onto each edge.
\end{restatable}

We remark that \cref{thm:outerplanar_bounded,thm:2tree_unbounded} together give a quite good understanding of which 2-trees have bounded stack number as stacking at most one vertex onto each edge yields exactly the maximal outerplanar DAGs.
In~$2006$, Di Giacomo, Didimo, Liotta and Wismath~\cite{DiGiacomo2006_SeriesParallel} asked whether or not all DAGs with treewidth 2 admit a $2$-stack layout.
(The DAG in \cref{fig:3-fence-example} is already a counterexample.)
<Most recently, N\"ollenburg and Pupyrev~\cite{Noellenburg2021_DAGsWithConstantStackNumber} highlight whether or not all directed acyclic $2$-trees have bounded stack number as an important open question.
They as well as a Dagstuhl report by Bekos et al.~\cite{Bekos2019_Planar_kPlanar} even conjecture that the stack number of all such $2$-trees should indeed be bounded.
Our \cref{thm:2tree_unbounded} refutes this conjecture.

Let us also emphasize that \cref{thm:outerplanar_bounded} in particular gives that upward outerplanar graphs have bounded stack number.
This was known only for specific subclasses before, such as internally triangulated upward outerpaths~\cite{Bhore2021_UpwardBookThickness}.
As such, \cref{thm:outerplanar_bounded} provides one of the largest known classes of upward planar graphs with bounded stack number, while it is a famous open problem whether or not actually all upward planar graphs have bounded stack number~\cite{Nowakowski1989_DirectedStackNumer}.

\subparagraph{Organization of the paper.}
Before proving \cref{thm:outerplanar_bounded} in \cref{sec:outerplanar} and \cref{thm:2tree_unbounded} in \cref{sec:2trees}, we define in \cref{sec:preliminaries} all concepts and notions relevant for our proofs.
This includes a quick reminder of $2$-trees, outerplanar graphs, stack layouts, and twist numbers, but also some specialized notions for directed acyclic $2$-trees, such as monotone and transitive vertices, block-monotone DAGs, or transitive subgraphs below monotone vertices.
Directed $H$-partitions are introduced in \cref{sec:directed-H-partitions}.
But first, let us review related work.

\subsection{Related Work}
\label{sec:related_work}

Stack layouts are just one type of so-called \emph{linear layouts} which have been an active field of study over at least the last forty years.
In full generality, a linear layout of an (undirected) graph~$G$ consists of a total ordering~$\prec$ of the vertices~$V(G)$ and a partition of the edges~$E(G)$ into parts such that each part fulfills certain combinatorial properties.
For directed acyclic graphs the vertex ordering~$\prec$ must be a topological ordering.
The most prominent types of linear layouts are \emph{stack layouts} (no two edges of the same part cross) and \emph{queue layouts} (no two edges of the same part nest, i.e., two edges~$ab$ and~$cd$ ordered $a \prec c \prec d \prec b$ are forbidden)~\cite{Heath1992_Queues,Heath1992_Comparing,Dujmovic2020_PlanarGraphsBoundedQueueNumber}.
These are accompanied by a rich field of further variants like \emph{mixed layouts} (a combination of stack and queue layouts)~\cite{Heath1992_Queues,Pupyrev2017_MixedLayoutPlanar,Angelini2022_MixedLayouts2Trees,Alam2022_MixedPageNumber,deCol2019_MixedLinearLayours,Foerster2023_BipartiteLinearLayouts}, \emph{track layouts}~\cite{Dujmovic2005_TrackLayouts,Bannister2019_TrackLayouts,Pupyrev2020_TrackNumber}, deque layouts~\cite{Auer2011_QueueDequeGraphs,Bekos2023_ComputeLinearLayouts,Bekos2023_DequeRiqueComplete}, rique layouts~\cite{Bekos2023_RiqueNumber,Bekos2023_ComputeLinearLayouts,Bekos2023_DequeRiqueComplete}, \emph{local} variants thereof~\cite{Merker2019_LocalStack,Merker2020_LocalQueue,Merker2021_LocalComplete,Angelini2022_MixedLayouts2Trees} and more~\cite{Dujmovic2004_LinearLayouts,Alam2021_DispersableBookEmbeddings}.

Let us content ourselves with just briefly summarizing below (some of) those previous results on stack layouts that concern undirected or directed planar graphs.

\subparagraph{Stack layouts of undirected planar graphs.}
Building on the notion of Kainen and Ollmann~\cite{Kainen1974_BookThickness,Ollmann1973_BookThickness}, the stack number for undirected graphs was first investigated by Bernhart and Kainen~\cite{Bernhart1979_BookThickness} in 1979.
Besides giving bounds for complete and complete bipartite graphs, they conjecture that there are planar graphs with arbitrarily large stack number.
This conjecture was refuted in~\cite{Buss1984_StacksPlanar,Heath1984_7StacksPlanar} leading eventually to an improved upper bound of~$4$~\cite{Yannakakis1989_4StacksPlanar}, which has only recently been shown to be tight~\cite{Bekos2020_4StacksPlanar,Yannakakis2020_4StacksPlanar}.

It is well-known that a graph with at least one edge admits a $1$-stack layout if and only if it is outerplanar, and that it admits a $2$-stack layout if and only if it is planar sub-Hamiltonian (i.e., is a subgraph of a planar graph containing a Hamiltonian cycle)~\cite{Bernhart1979_BookThickness}.
Testing whether a graph is planar sub-Hamiltonian is \NP-complete~\cite{Wigderson1982_HamiltonianCompletion}, but there has been significant effort to identify planar Hamiltonian and sub-Hamiltonian graph classes:
These include (among others) planar bipartite graphs~\cite{DeFraysseix1995_PlanarBipartite}, planar graphs with maximum degree at most~$4$~\cite{Bekos2016_Deg4Hamiltonian}, planar $4$-connected graphs~\cite{Tutte1956_4ConnectedHamiltonian}, planar $3$-connected graphs with maximum degree~$5$, and $2$-trees~\cite{Rengarajan1995_Stacks2Trees}.
Further, three stacks are sufficient for planar $3$-trees~\cite{Heath1984_7StacksPlanar} and planar graphs with maximum degree~$5$~\cite{Guan2020_Deg5Planar}.
Recall that four stacks are always sufficient~\cite{Yannakakis1989_4StacksPlanar} and sometimes necessary~\cite{Bekos2020_4StacksPlanar,Yannakakis2020_4StacksPlanar} for all planar graphs.

Among the numerous results on non-planar graphs $G$, there are bounds on $\sn(G)$ depending on the Euler genus~\cite{Malitz1994_EulerGenus}, the pathwidth~\cite{Togasaki2002_Pathwidth} and the treewidth~\cite{Ganley2001_Treewidth,Dujmovic2007_treewidth,Vandenbussche2009_kTrees} of $G$.

\subparagraph{Stack layouts of planar directed acyclic graphs.}
For directed acyclic graphs we additionally require the vertex ordering~$\prec$ to be a topological ordering.
Nowakowski and Parker~\cite{Nowakowski1989_DirectedStackNumer} were the first to study this and consider stack layouts of diagrams of posets.
Presenting an example with stack number~$3$ (which was later improved to~$4$ by Hung~\cite{Hung1993_Poset4Stacks} and to~$5$ by Merker~\cite{Merker2020_Thesis}), they ask whether posets with a planar diagram\footnote{In planar diagrams (generally: upward planar drawings) all edges must be drawn $y$-monotone along their edge direction. For example, the graph in \cref{fig:3-fence-example} is upward planar, while the graph $G_k$ in \cref{fig:planar-partial-3-tree} admits no upward planar drawing due to the edge from vertex $\ell_k$ to vertex $r_1$.} have bounded stack number.
Despite significant effort on different subclasses~\cite{Syslo1990_Posets,Alzohairi1997_SeriesParallelPosets,Heath1997_Posets,Heath1999_DAGs2,Heath1999_DAGs1,Alzohairi2001_Lattices,Alhashem2015_StacksPoset}, this question still remains open.

A slight generalization, namely whether all upward planar graphs have bounded stack number, is considered to be one of the most important open questions in the field of linear layouts~\cite{Frati2013_UpwardStackNumber,Noellenburg2021_DAGsWithConstantStackNumber,Jungeblut2022_SublinearUpperBound}.
It is known to hold for upward planar $3$-trees~\cite{Frati2013_UpwardStackNumber,Noellenburg2021_DAGsWithConstantStackNumber}.
However, this does not imply the same for upward planar DAGs of treewidth at most $3$, since upward planar partial $3$-trees might not have an upward planar $3$-tree as a supergraph.
(E.g., \cref{fig:3-fence-example} depicts such an example.)
Superseding previous results~\cite{Frati2013_UpwardStackNumber}, Jungeblut, Merker and Ueckerdt~\cite{Jungeblut2022_SublinearUpperBound} recently gave the first sublinear upper bound for all upward planar graphs by showing that every $n$-vertex upward planar graph $G$ has stack number $\sn(G) \leq O((n \log n)^{2/3}) = o(n)$.

For general planar DAGs (that are not necessarily upward planar), Heath, Pemmaraju and Trenk~\cite{Heath1999_DAGs1} show that directed trees admit $1$-stack layouts and directed unicyclic graphs admit $2$-stack layouts.
Other classes of DAGs admitting $2$-stack layouts include two-terminal series-parallel graphs~\cite{DiGiacomo2006_SeriesParallel}, $N$-free graphs~\cite{Mchedlidze2009_HamiltonianCompletion} or planar graphs whose faces have a special structure~\cite{Binucci2019_UpwardStackST}.
Recall from the example in \cref{fig:planar-partial-3-tree} that the stack number can be linear in the number of vertices, already for planar DAGs of treewidth~$3$.
With this in mind, Heath, Pemmaraju and Trenk~\cite{Heath1999_DAGs1} formulated \cref{conj:outerplanar} in~$1999$.

\cref{conj:outerplanar} has received considerable attention, especially in recent years.
Bhore, Da Lozzo, Montecchiani and N\"{o}llenburg~\cite{Bhore2021_UpwardBookThickness} confirm \cref{conj:outerplanar} for several subclasses of outerplanar DAGs, including internally triangulated upward outerpaths or cacti.
Subsequently, N\"{o}llenburg and Pupyrev~\cite{Noellenburg2021_DAGsWithConstantStackNumber} confirm \cref{conj:outerplanar} for single-source outerplanar DAGs, monotone outerplanar DAGs (to be defined in \cref{sec:preliminaries}) and general outerpath DAGs.  
In a Dagstuhl Report~\cite{Bekos2019_Planar_kPlanar}, Bekos et al.\ claim that every $n$-vertex outerplanar DAG $G$ has $\sn(G) \leq O(\log n)$, while they conjecture that actually $\sn(G) = O(1)$, even for all directed acyclic $2$-trees $G$.
\cref{thm:outerplanar_bounded} in the present paper confirms \cref{conj:outerplanar}, while \cref{thm:2tree_unbounded} refutes the conjecture of Bekos et~al.

Finally, let us mention that the decision problem of whether a given DAG admits a $k$-stack layout is known to be \NP-complete for every~$k \geq 2$~\cite{Bekos2022_DAGs2StacksNP,Binucci2019_UpwardStackST,Heath1999_DAGs2}.
However, there are FPT algorithms parameterized by the branchwidth~\cite{Binucci2019_UpwardStackST} or the vertex cover number~\cite{Bhore2021_UpwardBookThickness}.

\section{Preliminaries}
\label{sec:preliminaries}

All graphs considered here are finite, non-empty, and simple, i.e., have neither loops nor parallel edges.
For a graph $ G $ and an edge $ e $, we define $ G + e $, respectively $ G - e $, as the graph obtained from $ G $ by adding, respectively removing, the edge $ e $ if possible.
Similarly, vertices, sets of edges or vertices, and subgraphs can be added or removed with the same notation, where incident edges are also removed as necessary.

\subparagraph{Outerplanar graphs and 2-trees.}
We start by considering undirected graphs.
A graph~$G$ is \emph{outerplanar} if it admits a plane drawing with all vertices incident to the outer face.
Further, $G$ is \emph{maximal outerplanar} if no edge~$e$ can be added to~$G$ such that~$G + e$ remains outerplanar.
Equivalently, a maximal outerplanar graph is either a single vertex, a single edge, or consists of $n \geq 3$ vertices and admits a plane drawing whose outer face is bounded by a cycle of length $n$ and where every inner face is bounded by a triangle.

Outerplanar graphs are intimately related to graphs of treewidth~$2$ and so-called $2$-trees.
A \emph{$2$-tree} is inductively defined by the following construction sequence:
\begin{itemize}
    \item A single edge~$xy$ is a $2$-tree\footnote{
        Usually this inductive definition of $2$-trees starts with a triangle, but for the arguments below it is more convenient to let a single edge be a $2$-tree as well.
    }.
    This first edge in the process is called the \emph{base edge}.

    \item If~$G$ is a $2$-tree and~$vw$ an edge of~$G$, then the graph obtained from $G$ by adding a new vertex~$u$ and edges~$uv$ and $uw$ is again a $2$-tree.
    In this case we say that vertex~$u$ is \emph{stacked} onto edge~$vw$.
\end{itemize}
We remark that 2-trees are exactly the edge-maximal graphs of treewidth 2.
Note that the construction sequence of a $2$-tree~$G$ is not unique.
In fact, for every $2$-tree~$G$ and every edge~$xy$ of~$G$ there is a construction sequence of~$G$ with~$xy$ as the base edge.
(And even for a fixed base edge, there can be exponentially many construction sequences of~$G$.)
However, as soon as the base edge~$xy$ is fixed, this uniquely determines for each vertex~$u$ different from~$x$ and~$y$ the edge~$vw$ that~$u$ is stacked onto.
In this case we call~$vw$ the \emph{parent edge} of~$u$, vertices~$v$ and~$w$ the \emph{parents} of~$u$, and likewise~$u$ a \emph{child} of~$v$ and~$w$.
Note that for each edge~$vw$ in~$G$ different from the base edge, its endpoints~$v$ and~$w$ are in a parent/child relationship. 

Maximal outerplanar graphs are exactly those $2$-trees in which at most one vertex is stacked onto each edge, except for the base edge onto which up to two vertices can be stacked.
In fact, the base edge~$xy$ is an inner edge in every outerplanar drawing if two vertices are stacked onto~$xy$, and otherwise~$xy$ is an outer edge.
In the literature, maximal outerplanar graphs are also known as the \emph{simple $2$-trees}~\cite{Knauer2012_SimpleTreewidth,Wulf2016_SimpleTreewidth}.

Let~$G$ be a connected outerplanar graph, see also \cref{fig:block_construction_tree_DAG} for an example which is a subgraph of a simple $2$-tree (for now, ignore the edge orientations).
A \emph{block} $B$ of~$G$ is either a bridge or a maximal $2$-connected component, while a vertex $v$ of $G$ is a \emph{cut vertex} if $G-v$ is disconnected.
The \emph{block-cut tree} of $G$ has as vertex set all blocks and all cut vertices of $G$ and an edge between block $B$ and cut vertex $v$ if and only if $v \in B$.
See \cref{fig:block_tree} for a block-cut tree of the graph in \cref{fig:block_construction_tree_DAG}.

\begin{figure}
    \centering
    \begin{subfigure}[t]{0.3\textwidth}
        \centering
        \includegraphics[scale=1.30909090909, page=1]{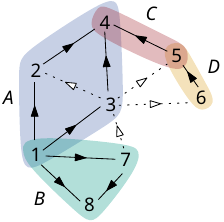}
        \caption{An outerplanar DAG~$G$.}
        \label{fig:block_construction_tree_DAG}
    \end{subfigure}
    \hspace{10em}
    \begin{subfigure}[t]{0.3\textwidth}
        \includegraphics[scale=1.30909090909, page=2]{figures/block_construction_tree}
        \caption{
            The block-cut tree of~$G$.
        }
        \label{fig:block_tree}
    \end{subfigure}
    \caption{
        An outerplanar DAG~$G$ with its block-cut tree.
        The construction sequence of~$G$ is $1, 2, \ldots, 8$.
        In particular, the base edge is from~$1$ to~$2$.
        Dotted edges are non-edges of $G$ whose addition to $G$ gives an outerplanar $2$-tree.
    }
    \label{fig:block_construction_tree}
\end{figure}

\subparagraph{Outerplanar DAGs and directed acyclic 2-trees.}
For the remainder of this paper we will exclusively consider \emph{directed graphs}, i.e., every edge~$e$ between two vertices~$v$ and~$w$ shall have a specified orientation, either from~$v$ to~$w$, or from~$w$ to~$v$.
In the former case we denote $e = vw$ and in the latter case $e = wv$.
In general, for every two disjoint vertex sets~$A$ and~$B$ we refer to the \emph{edges between~$A$ and~$B$} as those edges with exactly one endpoint in~$A$ and one endpoint in~$B$, regardless of their orientation.
On the other hand, the \emph{edges from~$A$ to~$B$} are those oriented from some vertex in~$A$ to some vertex in~$B$.
Notions like $2$-trees and outerplanar graphs are inherited to directed graphs from their underlying undirected graphs.
In particular, the \emph{blocks} of a directed graph are exactly the blocks of the underlying undirected graph.

A \emph{directed acyclic graph (DAG)} is a directed graph with no directed cycle, i.e., with no cycle $C = (v_1,\ldots,v_\ell)$, $\ell \geq 2$, with edges directed from $v_i$ to $v_{i+1}$ for $i=1,\ldots,\ell-1$ and from $v_\ell$ to $v_1$.
Consider a directed $2$-tree~$G$ with a fixed base edge~$xy$ (i.e., oriented from~$x$ to~$y$) and a vertex $u \neq x,y$ with parent edge~$vw$ (i.e., oriented from~$v$ to~$w$).
There are four possibilities for the directions of the edges between~$u$ and its parents~$v$ and~$w$ (see \cref{fig:stacking_types}):
\begin{itemize}
    \item If the edges~$wu$ and~$uv$ are in~$G$, then $(v,w,u)$ forms a directed cycle and we call~$u$ \emph{cyclic}.

    \item If the edges~$vu$ and~$uw$ are in~$G$, we call~$u$ \emph{transitive}.

    \item In the two remaining cases we call~$u$ \emph{monotone}.
    We say that~$u$ is a \emph{left child} of~$vw$ if $uv$ and~$uw$ are in~$G$, and that~$u$ is a \emph{right child} if~$vu$ and~$wu$ are in~$G$.
    Let us note that a left (right) child is to the left (right) of its parents in every topological ordering.
\end{itemize}
\begin{figure}[ht]
    \begin{subfigure}{0.15\textwidth}
        \centering
        \includegraphics[scale=1.30909090909, page=4]{figures/stacking_types.pdf}
        \caption{Cyclic.}
        \label{fig:stacking_types_cyclic}
    \end{subfigure}
    \hfill
    \begin{subfigure}{0.15\textwidth}
        \centering
        \includegraphics[scale=1.30909090909, page=1]{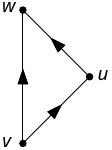}
        \caption{Transitive.}
        \label{fig:stacking_types_transitive}
    \end{subfigure}
    \hfill
    \begin{subfigure}{0.3\textwidth}
        \centering
        \includegraphics[scale=1.30909090909, page=3]{figures/stacking_types.pdf}
        \caption{Monotone (left child).}
        \label{fig:stacking_types_outgoing_monotone}
    \end{subfigure}
    \hfill
    \begin{subfigure}{0.3\textwidth}
        \centering
        \includegraphics[scale=1.30909090909, page=2]{figures/stacking_types.pdf}
        \caption{Monotone (right child).}
        \label{fig:stacking_types_incoming_monotone}
    \end{subfigure}
    \caption{The four possible stackings of a new vertex~$u$ onto a directed edge~$vw$.}
    \label{fig:stacking_types}
\end{figure}
Observe that the notions of cyclic, transitive and monotone vertices crucially depend on the choice of the base edge.
Further, observe that a directed $2$-tree~$G$ is a DAG if and only if no vertex is cyclic, independent of the choice of the base edge.
So if~$G$ is a DAG, then every construction sequence involves only transitive and monotone vertices.

We say that a directed $2$-tree~$G$ is \emph{monotone} (respectively \emph{transitive}) if there exists a choice for the base edge~$xy$ such that every vertex except for~$x$ and~$y$ is monotone (respectively transitive).
In particular, a monotone (or transitive) directed outerplanar graph is always a maximal directed outerplanar graph.
A connected outerplanar (but not necessarily maximal outerplanar) DAG is called \emph{block-monotone} if every block is monotone, see e.g., \cref{fig:block_monotone}.
\begin{csidefigure}
    \centering
    \includegraphics[scale=1.30909090909]{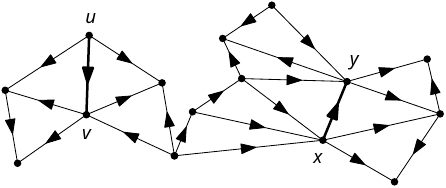}
    \caption{
        A block-monotone outerplanar DAG~$G$.
        Each of the two blocks of~$G$ contains a base edge ($xy$ and~$uv$) such that all other vertices are monotone.
    }
    \label{fig:block_monotone}
\end{csidefigure}

Given a maximal outerplanar DAG~$G$ with a fixed base edge~$xy$, its \emph{construction tree} with respect to~$xy$ is a rooted, undirected (and unordered) binary tree~$T$ on the vertices of~$G$ with vertex labels~\lab{M} (for \enquote{monotone}) and~\lab{T} (for \enquote{transitive}) such that\footnote{
    We remark that the construction tree (unrooted and without the labels) is exactly the tree of a nice tree-decomposition of width~$2$ of~$G$, but we do not use this fact here.
}:
\begin{itemize}
    \item The tail~$x$ of the base edge is the root and has label~\lab{M}.
    \item The head~$y$ of the base edge is the unique child of the root and also has label~\lab{M}.
    \item Whenever~$u$ is a child of an edge~$vw$ or $wv$ of~$G$ with~$w$ being a child of~$v$, then in $T$ we have that~$u$ is a child of~$w$.
    Moreover, vertex $u$ is labeled~\lab{M} in $T$ if~$u$ is a monotone vertex and~\lab{T} if~$u$ is a transitive vertex.
\end{itemize}
Observe that each vertex in the construction tree has at most two children.
For each vertex~$v$, the \emph{transitive subgraph below $v$} is the subgraph of~$G$ induced by~$v$ and all its descendants~$w$ in~$T$ such that the unique $v$-$w$-path in~$T$ consists solely of vertices labeled~$\lab{T}$ (except possibly~$v$ itself).
See \cref{fig:transitive_subgraphs} for a construction tree and the transitive subgraphs below monotone vertices.

\begin{csidefigure}
    \centering
    \includegraphics[scale=1.30909090909]{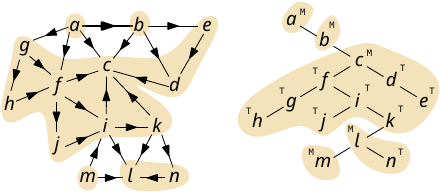}
    \caption{An outerplanar graph with base edge $ ab $ and its construction tree. The transitive subgraph below each monotone vertex is highlighted orange.}
    \label{fig:transitive_subgraphs}
\end{csidefigure}

\subparagraph{Stack number versus twist number.}

Recall that if in a given vertex ordering~$\prec$ of~$G$ we have a set of~$k$ pairwise crossing edges, we call this a $k$-twist in~$\prec$.
The maximum~$k$ such that there is a $k$-twist in~$\prec$ is called the \emph{twist number} of~$\prec$.
Then clearly with respect to this vertex ordering, $E(G)$ cannot be partitioned into fewer than~$k$ stacks.
The minimum twist number over all (topological) vertex orderings~$\prec$ of~$G$ is called the \emph{twist number}~$\tn(G)$ of~$G$.
Hence it follows that $\sn(G) \geq \tn(G)$ for every DAG~$G$, i.e., having large twists in every topological vertex ordering is a simple reason for having a large stack number.
Somewhat surprisingly, a large twist number is the only reason for a large stack number, up to a polynomial function.

\begin{theorem}[Davies, 2022~\cite{Davies2022_ColoringCircleGraphs}]
    \label{thm:stack_in_terms_of_twist}
    For every vertex ordering~$\prec$ of~$G$ with twist number~$k$, we can partition~$E(G)$ into $2k\log_2(k)+2k\log_2(\log_2(k)) + 10k$ stacks.
\end{theorem}

In fact, Davies~\cite{Davies2022_ColoringCircleGraphs} gives an upper bound on the chromatic number~$\chi(H)$ of a circle graph~$H$ in terms of its clique number~$\omega(H)$.
(Gy\'{a}rf\'{a}s~\cite{Gyarfas1985_CircleGraphsChiBounded} was the first to show that circle graphs are $\chi$-bounded, but Davies gives the first asymptotically tight bound.)
To obtain \cref{thm:stack_in_terms_of_twist}, simply consider the circle graph~$H$ with $V(H) = E(G)$ whose edges correspond to crossing  edges in~$\prec$.
Then~$\omega(H)$ is the twist number of~$\prec$ and a proper $k$-coloring of~$H$ is a partition of~$E(G)$ into~$k$ stacks with respect to~$\prec$.

\section{Outerplanar DAGs have Constant Stack Number}
\label{sec:outerplanar}

In this section we prove \cref{thm:outerplanar_bounded}, our first main result.
One key ingredient is a recent result by N\"{o}llenburg and Pupyrev~\cite{Noellenburg2021_DAGsWithConstantStackNumber} stating that the stack number of monotone outerplanar graphs is bounded.
The idea behind our approach is to partition a given outerplanar DAG~$G$ into \enquote{transitive parts}, such that the contraction of each part into a single vertex yields a block-monotone DAG~$H$.
Then the result from~\cite{Noellenburg2021_DAGsWithConstantStackNumber} can be applied to the blocks of~$H$ individually.
Two things are left to do:
First we show that the many stack layouts for the blocks of~$H$ can be combined into a single stack layout of $H$ without requiring too many additional stacks.
Then we show that each transitive part can be \enquote{decontracted} to yield a stack layout of $G$, again without requiring too many additional stacks.

We formalize all this by introducing a novel structural tool called \emph{directed \Hpartition{}s}.

\subsection{Monotone and Block-Monotone Outerplanar DAGs}

N\"{o}llenburg and Pupyrev~\cite{Noellenburg2021_DAGsWithConstantStackNumber} analyzed the stack number of different subclasses of outerplanar DAGs.
One of their results is that monotone outerplanar DAGs with at most one vertex stacked on the base edge\footnote{
    The definition of \emph{monotone} in~\cite{Noellenburg2021_DAGsWithConstantStackNumber} allows only one vertex stacked on the base edge.
    This is in contrast to our definition where the base edge is (the only edge) allowed to have two vertices stacked onto it.
} have bounded twist number (and therefore also bounded stack number).

\begin{theorem}[N\"{o}llenburg, Pupyrev \cite{Noellenburg2021_DAGsWithConstantStackNumber}]
    \label{thm:monotone_outerplanar_twist_4}
    Every monotone outerplanar DAG~$G$ with at most one vertex stacked on the base edge has twist number~$\tn(G) \leq 4$.
\end{theorem}

\begin{corollary}
    \label{cor:monotone_outerplanar_stack_128}
    Every monotone outerplanar DAG~$G$ has stack number~$\sn(G) \leq 128$.
\end{corollary}

\begin{proof}
    Let~$G = (V,E)$ be a monotone outerplanar DAG with base edge~$e$.
    Recall that the base edge of~$G$ may have two children so~$ G $ is the union of at most two monotone outerplanar DAGs~$G_1$ and~$G_2$ such that both use~$e$ as their base edge, both have at most one vertex stacked on~$e$, and their intersection is exactly~$e$.
    By \cref{thm:monotone_outerplanar_twist_4} for $i=1,2$ graph $G_i$ admits a topological ordering $\prec_i$ with twist number at most~$4$.
    Thus by \cref{thm:stack_in_terms_of_twist} we have
    \[
        \sn(G_i) \leq 2\cdot 4 \cdot \log_2(4) + 2\cdot 4 \cdot \log_2(\log_2(4)) + 10\cdot 4 = 64.
    \]
    As the endpoints of~$e$ appear in the same order in~$\prec_1$ and~$\prec_2$, the corresponding $64$-stack layouts of $G_1$ and $G_2$ can be combined into a $128$-stack layout of $G$.
\end{proof}

The next \lcnamecref{lem:bounded_by_blocks} and the following \lcnamecref{cor:monotone_outerplanar_stack_128} extend the bound for monotone DAGs to those that are block-monotone.
This will be important later, as block-monotonicity plays a crucial role in our proof that outerplanar DAGs have constant stack number.

\begin{lemma}
    \label{lem:bounded_by_blocks}
    Let~$G$ be a DAG and~$\calB$ be the set of its blocks.
    Then we have
    \[
        \sn(G) \leq 2 + 2 \cdot \max\limits_{B \in \calB} \sn(B)
        \text{.}
    \]
\end{lemma}

\begin{proof}
    We may assume that~$G$ is connected.
    Let~$T$ be the block-cut tree of~$G$ rooted at an arbitrary block of~$G$ and let $s = \max\{\sn(B) \mid B \in \calB\}$ be the maximum stack number among all blocks of~$G$.
    We incrementally construct a stack layout of~$G$ by processing the blocks of $G$ according to their level in~$T$ one after another; the level of a block $B$ being the number of cut vertices on the path from $B$ to the root in~$T$.
    In doing so, we maintain the following two invariants:
 
    \begin{enumerate}[({I}1)]
        \item\label{inv:total_stacks}
        At most~$2(s+1)$ stacks are used in total.
        \item\label{inv:block_stacks}
        At most~$s+1$ stacks are used for each block.
    \end{enumerate}

    We start with the root block, which (like all other blocks) admits an $s$-stack layout by assumption.
    This fulfills the invariants~\invref{inv:total_stacks} and~\invref{inv:block_stacks} trivially.

    Now consider a block~$B$ in level~$\ell \geq 0$.
    We assume that $ B $ is already laid out and insert its children $B_1, \ldots, B_k$ into the layout.
    For this, repeat the following for every cut vertex~$v$ that~$B$ shares with blocks $B_1, \ldots, B_k$ in level~$\ell+1$.
    For $i = 1, \ldots, k$ take an $s$-stack layout of~$B_i$ (which exists by assumption) and let~$L_i$ and~$R_i$ denote the sets of vertices to the left of~$v$ and to the right of~$v$ in the $s$-stack layout for~$B_i$, respectively.
    We insert the layouts of $B_1, \ldots, B_k$ directly to the left and directly to the right of~$v$ such that the vertices appear in the following order (see also \cref{fig:block_stacks}):
    \begin{sidefigure}
        \centering
        \includegraphics[scale=1.30909090909]{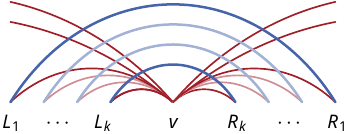}
        \caption{
            Integration of the stack layouts of blocks~$B_1, \ldots, B_k$ around the cut vertex~$v$ in the current partial stack layout.
            Blocks~$B_1, \ldots, B_k$ can use the same set of~$s$ stacks except for all edges incident to~$v$ which are on the $(s+1)$-th stack.
        }
        \label{fig:block_stacks}
    \end{sidefigure}
    \[
        L_1 \prec \cdots \prec L_k \prec
        v \prec
        R_k \prec \cdots \prec R_1
    \]

    Let~$E'$ denote the set of all edges that are in $B_1, \ldots, B_k$.
    An edge from~$E'$ and an edge~$e$ from a smaller level can cross only if~$e$ is incident to~$v$ and therefore belongs to~$B$.
    As the edges of~$B$ use at most~$s+1$ stacks by invariant~\invref{inv:block_stacks}, there are another~$s+1$ stacks available for the edges in~$E'$.
    To assign the edges of~$E'$ to stacks we start with the $s$-stack layouts of each block and then move all edges incident to~$v$ to the $(s+1)$-th stack.
    Observe that edges in~$E'$ belonging to different blocks can only cross if exactly one of them is incident to~$v$.
    As the edges incident to~$v$ form a star centered at~$v$, we conclude that all stacks are crossing-free and~$s+1$ stacks indeed suffice for~$E'$, maintaining~\invref{inv:block_stacks}.

    Finally, obverse that children of different cut vertices are separated in the layout and thus their edges do not cross.
    Therefore, we have a $(2s+2)$-stack layout for all blocks that are already completed, maintaining~\invref{inv:total_stacks}.
\end{proof}

Now \cref{cor:monotone_outerplanar_stack_128,lem:bounded_by_blocks} immediately imply the following.

\begin{corollary}
     \label{cor:block_monotone_bounded}
     Every block-monotone outerplanar DAG admits a $258$-stack layout.
\end{corollary}

\subsection{Directed \texorpdfstring{$\boldsymbol{H}$}{H}-Partitions}
\label{sec:directed-H-partitions}

We introduce \emph{directed \Hpartition{}s} as a new structural tool and explore how they can be applied to reason about stack layouts of DAGs.
They are not limited to DAGs nor to treewidth~$2$, and we believe they may be applicable also for other graph classes in the future.
A related tool known as \emph{layered \Hpartition{}s} was introduced in~\cite{Dujmovic2020_PlanarGraphsBoundedQueueNumber} to establish that undirected planar graphs have bounded queue number and has since been used widely. 
Thus, we state and prove the lemmas in this section in a more general form than we actually need them.
The construction of a directed \Hpartition for outerplanar DAGs is deferred to \cref{sec:construct_H_partition}.
The definition of directed \Hpartition{}s is illustrated in \cref{fig:directed-H-partition}.

\begin{figure}
    \centering\includegraphics[scale=1.30909090909]{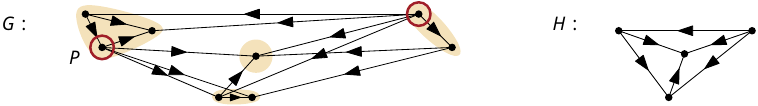}
    \caption{A directed \Hpartition (orange) of a directed graph $ G $ (left) and the quotient $ H $ (right). The cut cover number of the part $ P $ is $ 2 $ as the two vertices marked in red cover all edges with exactly one endpoint in $ P $.}
    \label{fig:directed-H-partition}
\end{figure}

\begin{samepage}
\begin{definition}[Directed $H$-Partition, Cut Cover Number]\label{def:directed-H-partition}
    Let~$G$ and~$H$ be directed graphs.
    A \emph{directed \Hpartition} of~$G$ is a partition $\calP$ of~$V(G)$ such that the following holds:
    \begin{itemize}
        \item For every two parts~$P,Q \in \calP$ the edges of~$G$ 
        between $P$ and $Q$
        are oriented either all from~$P$ to~$Q$ or all from~$Q$ to~$P$.
        \item The quotient $G/\calP$ is isomorphic to~$H$.
        Here~$G/\calP$ is obtained from~$G$ by contracting each part $P \in \calP$ into a single vertex $v_P$ and directing an edge from $v_P$ to $v_Q$ in $H$ whenever in $G$ there is some edge from $P$ to $Q$.
        (This orientation is well-defined by the first property.)
    \end{itemize}
    For a part $P \in \calP$ its \emph{cut cover number} is the smallest number of vertices of~$G$ required to cover (i.e., be incident to) all edges of~$G$ with exactly one endpoint in~$P$.
    The cut cover number of~$\calP$ is the maximum cut cover number among all its parts.
\end{definition}
\end{samepage}

So each vertex $v_P$ in $H$ corresponds to a part in $P \in \calP$ and a subgraph of $G$, denoted by $ G[v_P] = G[P] $, that is induced by the vertices in $P$.
More generally, every induced subgraph $B$ of $H$ corresponds to a subset $\calP_B \subseteq \calP$ of parts, and we let $G[B]$ denote the corresponding subgraph of $G$ that is induced by all vertices of $G$ contained in parts in $\calP_B$.

The definition of directed \Hpartition{}s is very similar to the well-known concept of (undirected) \Hpartition{}s.
The main difference and difficulty is that we need to ensure that the quotient is well-defined, i.e., the orientation of the edges between two parts is consistent.
Nevertheless, many useful properties of the undirected version are inherited.
In particular, if every part is connected, then the quotient is a minor of the underlying graph, which implies that treewidth and planarity are preserved.
Although not used here, we remark that the successful idea of Dujmović et al.~\cite{Dujmovic2020_PlanarGraphsBoundedQueueNumber} to combine \Hpartition{}s with layerings to so-called layered \Hpartition{}s is also feasible in the directed setting.

\newpage
Given directed graphs\footnote{Neither \cref{def:directed-H-partition} nor the definition of expanding vertex orderings requires $G$ or $H$ to be acyclic. However, in this paper we shall consider solely constellations where $G$ and $H$ are both DAGs.} $G$ and $H$ and a directed \Hpartition~$\calP$ of $G$, we say that a vertex ordering $\prec_G$ of~$G$ \emph{expands} a vertex ordering $\prec_H$ of~$H$ if all vertices of~$G$ belonging to the same part of~$\calP$ appear consecutively in~$\prec_G$ and whenever $P \in \calP$ lies to the left of $Q \in \calP$ in $\prec_G$, then $v_P \in V(H)$ lies to the left of $v_Q \in V(H)$ in $\prec_H$.

\begin{lemma}
\label{lem:bounded_H_partition_to_stack_layout}
    Let~$G$ and~$H$ be DAGs and~\calP be a directed \Hpartition of~$G$ with cut cover number at most~$w$.
    Further, let $\sn(G[P]) \leq s$ for each $P \in \calP$.
    Then for every $h$-stack layout $\prec_H$ of~$H$, there is a $(3wh + s)$-stack layout $\prec_G$ of~$G$ expanding $\prec_H$.
    
    In particular, we have $\sn(G) \leq 3w \cdot \sn(H) + s$.    
\end{lemma}

\begin{proof}
    We expand a given~$h$-stack layout $\prec_H$ of~$H$ to a $(3wh + s)$-stack layout $\prec_G$ of~$G$.
    For each part $P \in \calP$ consider an $s$-stack layout~$\prec_P$ of $ G[P] $, which exists by assumption,
    and replace in~$\prec_H$ vertex~$v_P$ corresponding to $P$ by the vertices in~$P$ ordered as in~$\prec_P$.
    As in the resulting vertex ordering $\prec_G$ of~$G$ all vertices from the same part appear consecutively (in other words, $\prec_G$ expands $\prec_H$), it follows that no two edges~$e_1, e_2$ in $G$ belonging to different parts in $\calP$ cross.
    Therefore, we may assign all edges with both endpoints in the same part to the same set of~$s$ stacks.

    It remains to assign the edges of~$G$ with endpoints in two different parts to the remaining~$3wh$ stacks.
    For this we consider each stack $S$ in the $h$-stack layout of~$H$ separately and show that all edges of~$G$ 
    corresponding to edges in~$S$
    can be assigned to~$3w$ stacks.
    (An edge $vw \in E(G)$ with $v \in P$ and $w \in Q$ corresponds to the edge~$v_Pv_Q \in E(H)$.)
    First, we note that the edges in~$S$ form an outerplanar subgraph of~$H$.
    As such, it can be partitioned into three star forests~\cite{Hakimi1996_StarAboricity}.
    Again, we can treat each star forest separately, and we are left with assigning the edges in~$G$ corresponding to the same star forest $F$ in $H$ to at most~$w$ stacks.

    For this, consider two edges~$e_1, e_2$ of~$G$ corresponding to the same star forest~$F$ that cross in our chosen vertex ordering $\prec_G$ of~$G$.
    If their endpoints are in four different parts in~$\calP$, then their corresponding edges in~$H$ cross, which is impossible.
    Therefore, the endpoints of~$e_1$ and~$e_2$ lie in at most three different parts.
    Hence, there is one part containing at least two of the endpoints, and we conclude that~$e_1$ and~$e_2$ actually correspond to the same star in $F$.
    Thus it suffices to consider the stars of~$F$ separately and reuse the same set of~$w$ stacks for all stars from~$F$.
    Now consider all edges of $G$ corresponding to the same star~$X$ of~$F$.
    As the cut cover number of~$\calP$ is at most~$w$, there is a set~$V' \subseteq V(G)$ with~$|V'| \leq w$ such that every edge of~$G$ corresponding to~$X$ is incident to (at least) one vertex in~$V'$.
    For each vertex $v' \in V'$ its incident edges form a star in $G$ and are therefore non-crossing in $\prec_G$.
    Thus, we can assign all incident edges at $v'$ corresponding to $X$ to the same stack.
    
    This requires at most $w$ stacks for each star $X$ in $F$, hence also at most $w$ stacks for each star forest $F$ of~$S$. To sum up, we have at most $3w$ stacks for each stack $S$ of the $h$-stack layout of $H$.
    Including the $s$ stacks from the beginning, this yields at most $3wh + s$ stacks in total.
\end{proof}

\cref{lem:bounded_H_partition_to_stack_layout} gives a good stack layout of $G$, provided $G$ admits a directed \Hpartition with small cut cover number for some $H$ with small stack number.
The notion of the cut cover number enables us to give the bound on the stack number independently of the size of the parts.
We remark that without a bound on the cut cover number, there may be a twist between the vertices of two parts that is as large as the smaller of the two parts.
For the next \lcnamecref{lem:H_partition_to_stack_layout}, we loosen the prerequisites by considering the blocks of $ H $ separately.
First, we require for each block $B$ of $H$ that the corresponding subgraph $G[B]$ has a good stack layout (for example, due to a small cut cover number of the inherited directed \Hpartition of $G[B]$).
And second, the interactions between the blocks of $H$ sharing a common cut vertex $v$ are restricted.

\begin{restatable}{lemma}{HPartitionToStackLayout}
    \label{lem:H_partition_to_stack_layout}\RestateRemark
    Let~$G$ be a DAG with a directed \Hpartition $\calP$ such that 
    \begin{itemize}
        \item for every block~$B$ of~$H$ the subgraph~$G[B]$ of~$G$ admits an $s$-stack layout expanding some vertex ordering of $H$.
    \end{itemize}
    Moreover, let $T$ be the block-cut tree of~$H$ rooted at some block of $H$, such that for every cut vertex~$v$ of~$H$ with
    child blocks~$C_1, \ldots, C_k$ in~$T$ the following holds:
    \begin{itemize}
        \item For $i=1,\ldots,k$, the intersection of $G[v]$ with the neighborhood of $ G[C_i - v]$ consists of a single edge $e_i \in E(G[v])$. 
        \item Edges $e_1,\ldots,e_k$ can be covered with at most $p$ directed paths in $G[v]$.
        \item For each edge $e \in E(G[v])$ we have $e = e_i$ for at most $t$ indices $i \in \{1,\ldots,k\}$.
    \end{itemize}
    Then~$\sn(G) \leq 4spt$.
\end{restatable}

\begin{figure}
    \includegraphics[scale=1.30909090909]{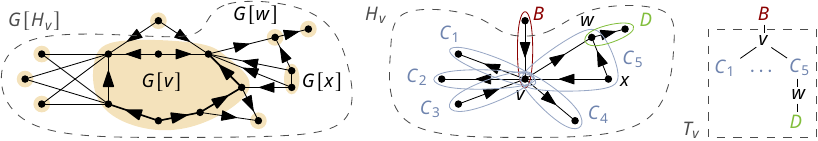}
    \caption{%
        Left: A DAG $ G $ with a directed \Hpartition as in \cref{lem:H_partition_to_stack_layout} with $ p = 2 $ and $ t = 3 $. Some edge directions are omitted for better readability.
        Middle: $ H $ and its blocks.
        Right: The block-cut tree of $ H $.
    }
    \label{fig:H_partition_to_stack_layout}
\end{figure}

We refer to \cref{fig:H_partition_to_stack_layout} for an illustration of the situation in \cref{lem:H_partition_to_stack_layout}.
In the following proof, we denote the neighborhood of a subgraph $ G' $ of some graph $ G $ by $ N(G') $.

\begin{proof}
    Let $v$ be a vertex in $H$.
    Recall that $v = v_P$ represents a part $P \in \calP$ in the directed \Hpartition and $G[v] \subseteq G$ is an induced subgraph of $G$.
    If $v$ is a cut vertex of $H$, we associate to $v$ also another subgraph of $G$ by considering everything below $v$ in the block-cut tree $T$.
    Formally, let $T_v$ denote the subtree of $T$ with root $v$ and let $H_v$ denote the subgraph of $H$ that is the union of all blocks in $T_v$.
    Then the corresponding subgraph $G[H_v]$ is the subgraph of $G$ induced by the union of all parts $P \in \calP$ for which the vertex $v_P$ in $H$ appears in some block $B$ in $T_v$.
    See again \cref{fig:H_partition_to_stack_layout} for an illustration of the notation.
    
    We shall find a $4spt$-stack layout of $G$ whose vertex ordering $\prec_G$ has the following properties:
    
    \begin{enumerate}[({I}1)]
        \item\label{inv:regions_below_cut_vertices}
        For every cut vertex $v$ of $H$ the vertices in $G[H_v]$ appear consecutively in $\prec_{G}$.
        \item\label{inv:regions_below_non-cut_vertex}
        For every non-cut vertex $v$ of $H$ the vertices in $G[v]$ appear consecutively in $\prec_{G}$.
    \end{enumerate}

    Assuming $\prec_G$ satisfies \invref{inv:regions_below_cut_vertices}, the following holds:
    
   \begin{claim}\label{clm:disjoint_block_do_not_cross}
      For every two vertex-disjoint blocks $B_1,B_2$ of $H$, no edge in $G[B_1]$ crosses an edge in $G[B_2]$ with respect to $\prec_G$.
   \end{claim}

    \begin{subproof}
        Let $T_1$ and $T_2$ be the subtrees of $T$ rooted at $B_1$ and $B_2$, respectively.
        First assume that $T_1$ and $T_2$ are disjoint, which in particular means that neither $B_1$ nor $B_2$ is the root of $T$.
        With $v_1, v_2$ being the parents of $B_1, B_2$ in $T$, respectively, we have $v_1 \neq v_2$ since $B_1$ and $B_2$ are vertex-disjoint.
        Then \invref{inv:regions_below_cut_vertices} gives that vertices of $G[B_1] \subseteq G[H_{v_1}]$ and $G[B_2] \subseteq G[H_{v_2}]$ appear in $\prec_G$ in disjoint intervals.
        Thus no edge in $G[B_1]$ crosses an edge in $G[B_2]$.

        If $T_1$ and $T_2$ are not vertex-disjoint, assume without loss of generality that $T_1 \subset T_2$; in particular that $B_1$ is not the root of $T$.
        Then by \invref{inv:regions_below_cut_vertices} the vertices of $G[H_{v_1}]$ form in $\prec_G$ a contiguous interval $I$.
        In particular, every edge in $G[B_1] \subseteq G[H_{v_1}]$ has both endpoints in $I$.
        As $B_1$ and $B_2$ are vertex-disjoint, we have $G[B_2] \cap G[H_{v_1}] = \emptyset$ and every edge in $G[B_2]$ has neither endpoint in $I$.
        Thus no edge in $G[B_1]$ crosses an edge in $G[B_2]$.
    \end{subproof}

    \Cref{clm:disjoint_block_do_not_cross} allows us to reuse the same set of stacks for vertex-disjoint blocks.
    With this in mind, we partition the blocks of $H$ into two sets~$\calB_\text{odd}$ and~$\calB_\text{even}$ containing the blocks with an odd, respectively even number of cut vertices on their path to the root in $T$.
    Then it is enough to use a set of $2spt$ stacks for blocks in $\calB_\text{odd}$ and a set of $2spt$ different stacks for $\calB_\text{even}$, giving the desired $4spt$ stacks in total.
    Observe that within the same set of blocks, say $\calB_\text{odd}$, two blocks are either again vertex-disjoint (and thus non-crossing by \cref{clm:disjoint_block_do_not_cross}) or have a common parent in~$T$.
    Thus it is left to consider a single cut vertex and its child blocks.

    We now construct the desired $4spt$-stack layout $\prec_G$ of $G$ by processing the block-cut tree~$T$ from the root to the leaves.
    After initializing the root block, in each step we consider a cut vertex $v$ whose parent block $B$ is already processed and process all child blocks of $v$ simultaneously.
    In doing so, we maintain after each step \invref{inv:regions_below_cut_vertices} and \invref{inv:regions_below_non-cut_vertex}
    for the already processed subgraph of $G$.
    To initialize, the root of $T$ is a single block of $H$ and admits an $s$-stack layout expanding some vertex ordering $\prec_H$ of $H$ by assumption.
    This fulfills \invref{inv:regions_below_cut_vertices} trivially and \invref{inv:regions_below_non-cut_vertex} since the layout expands $\prec_H$.
    
    Now for a step, consider a cut vertex~$v$ whose parent block is already processed and let $C_1,\ldots,C_k$ be the child blocks of $v$ in $T$.
    By the assumptions of the \lcnamecref{lem:H_partition_to_stack_layout}, for each~$i=1,\ldots,k$ the intersection $G[v] \cap N(G[C_i - v])$ consists of a single edge $e_i \in E(G[v])$, the edges $e_1,\ldots,e_k$ are covered with at most~$p$ directed paths $Q_1,\ldots,Q_p$ in~$G[v]$, and for each edge $e \in E(G[v])$ we have $e = e_i$ for at most~$t$ indices $i \in \{1,\ldots,k\}$.
    Since all of $C_1,\ldots,C_k$ are in $\calB_\text{even}$ or all in $\calB_\text{odd}$, we have a set of $2spt$ stacks at our disposal.
    Reserve~$p$ pairwise disjoint sets of stacks of size~$2st$, one per path.
    For a fixed path $Q$, group these $2st$ stacks further into~$t$ disjoint subsets of size~$2s$, such that for each subset each edge $e$ of $Q$ corresponds to at most one index $i \in \{1,\ldots,k\}$ with $e = e_i$.
    It is left to show that we can find a $2s$-stack layout for a fixed path~$Q = (x_0, \ldots, x_\ell)$ of length~$\ell$ and a set of blocks~$X_1,\ldots,X_\ell$, where $G[v] \cap N(G[X_j-v])$ is exactly the edge~$x_{j-1}x_j$ with $ j = 1, \dots, \ell $ (having one block per edge is the most difficult case, having less only makes it easier), refer to \cref{fig:h_partition_to_stack_layout_simplified}.
    Note that each child $ C_i $ is dealt with: It is either in $\calB_\text{even}$ or $\calB_\text{odd}$, it is attached to one of the $ p $ paths, and it is contained in one of the $ t $ subsets for the path within which $ C_i $ does not share the edge $ e_i $ with some other child. 
    That is, each of the $ X_j$'s takes care of at most one $ C_i $.

    \begin{figure}
        \centering
        \includegraphics[scale=1.30909090909, page=2]{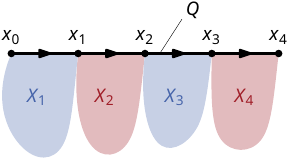}
        \caption{%
            The situation of \cref{lem:H_partition_to_stack_layout} (see \cref{fig:H_partition_to_stack_layout}, in particular the three blocks $ C_1, C_2, C_3 $ attached to the same edge in $ G[v] $) simplified in two ways:
            We only consider one directed path $ Q \subseteq G[v] $ and we have exactly one block $ X_j $ at each edge $ x_{j-1} x_j $, instead of at most~$ t $ (which is paid for by a factor of $ t $) or none (for which the respective $ X_j $ can simply be ignored).
            That is, each block $ X_j $ represents at most one child block $ C_i $ of $ v $, where $ e_i = x_{j-1} x_j $.
            Note that the colors match those used in the layout in \cref{fig:h_partition_stacks}.
        }
        \label{fig:h_partition_to_stack_layout_simplified}
    \end{figure}
    
    We are now in the situation illustrated in \cref{fig:h_partition_to_stack_layout_simplified} and by the assumptions of the \lcnamecref{lem:H_partition_to_stack_layout}, for each block $X_j$ the corresponding $G[X_j]$ admits an $s$-stack layout $\prec_j$ expanding some vertex ordering of $H$.
    For every vertex $w$ in $X_j$ the vertices of $G[w]$ appear consecutively in $\prec_j$, and this holds in particular for $w = v$.
    We remove from $\prec_j$ all vertices in $G[v]$ except for $x_{j-1}$ and $x_j$.
    As $G[v] $ and $ G[X_j - v] $ are connected only via $x_{j-1}$ and $x_j$, this does not remove any edge of $G$ corresponding to an edge in $X_j - v$.

    \begin{cfigure}
        \centering
        \includegraphics[scale=1.30909090909]{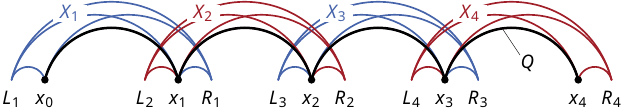}
        \caption{
            Incorporating the $s$-stack layouts of $G[X_1],\ldots,G[X_4]$ (compare \cref{fig:h_partition_to_stack_layout_simplified}) into the interval containing $G[v]$, where $G[v] \cap N(G[X_i-v]) = \{x_{i-1},x_i\}$ and $Q = (x_0,\ldots,x_4)$ is a directed path in $G[v]$.
            Both, the red and the blue edges represent $s$-stack layouts of respective $G[X_i]$.
        }
        \label{fig:h_partition_stacks}
    \end{cfigure}
    
    Let~$L_j$ and~$R_j$ denote the sets of vertices that are to the left of~$x_{j-1}$ respectively to the right of~$x_j$ in the remaining $\prec_j$.
    Recall that there are no vertices between $ x_{j-1} $ and $ x_j $ as the vertices of $ G[v] $ appear consecutively in $ \prec_j $.
    We now insert $L_j$ immediately before $x_{j-1}$ and $R_j$ immediately after $x_j$ in the vertex ordering of the already processed graph.
    See also \cref{fig:h_partition_stacks} for a visualization.
    We observe that the edges of~$X_j$ and~$X_{j'}$ do not cross for~$|j - j'| > 1$.
    Thus, $2s$ stacks indeed suffice for all~$X_1,\ldots,X_\ell$ by reusing the same $s$ stacks for even indices $j$ and another $s$ stacks for odd indices $j$.

    To finish the proof, recall that we maintain \invref{inv:regions_below_cut_vertices} and \invref{inv:regions_below_non-cut_vertex} after each step for the already processed subgraph of $G$.
    For this, note that by processing its child blocks, $ v $ is turned from a non-cut vertex to a cut vertex, and thus we may assume \invref{inv:regions_below_non-cut_vertex} but need to satisfy \invref{inv:regions_below_cut_vertices} for~$ v $.
    As, by \invref{inv:regions_below_non-cut_vertex}, the vertices of $G[v]$ were consecutive in the vertex ordering before, it follows that the vertices of $G[H_v]$ are consecutive in the vertex ordering after those insertions, i.e., \invref{inv:regions_below_cut_vertices} is fulfilled.
    Moreover, since each vertex ordering $\prec_j$ expands some vertex ordering of $H$, we conclude that for each newly processed vertex $w$ of $H$, its corresponding subgraph $G[w]$ lies consecutively inside $L_j$ or inside $R_j$ for some $j$.
    Thus also \invref{inv:regions_below_non-cut_vertex} is fulfilled.
%
\end{proof}

\begin{remark}
     The assumption in \cref{lem:H_partition_to_stack_layout} that the edges $e_1,\ldots,e_k$ can be covered with at most $p$ directed paths in $G[v]$ can be easily relaxed.
     Indeed, if $B$ is the parent block of $v$ in the block-cut tree of $H$, it is enough that the $s$-stack layout of $G[v]$ as part of the $s$-stack layout of $G[B]$ in the statement of the \lcnamecref{lem:bounded_H_partition_to_stack_layout} contains a set of at most $m$ non-crossing matchings that cover $e_1,\ldots,e_k$.
     Then the above proof can be easily adapted to show $\sn(G) \leq 2smt$.
     Having at most $p$ directed paths is clearly enough to have $m \leq 2p$ non-crossing matchings in \textit{every} topological vertex ordering of $G[v]$.
\end{remark}

\subsection{Directed \texorpdfstring{$\boldsymbol{H}$}{H}-Partitions of Acyclic Outerplanar Graphs}
\label{sec:construct_H_partition}

The goal of this section is to construct directed \Hpartition{}s $\calP$ for every outerplanar DAG $G$, such that we can apply \cref{lem:bounded_H_partition_to_stack_layout} from the previous section to reason that $\sn(G)$ is bounded by a constant.
In particular, we aim for a block-monotone~$H$, with each part $P \in \calP$ inducing a relatively simple subgraph in $G$, as well as small cut cover numbers.
Instead of bounding the cut cover number globally, it suffices to have it constant for each block of~$H$ locally.
Formally, if $B$ is a block of $H$, then $\calP_B = \{P \in \calP \mid v_P \in B\}$ is a directed $B$-partition of the corresponding subgraph $G[B]$ of $G$, and we want that the cut cover number of $\calP_B$ is constant.
This is enough to apply~\cref{lem:H_partition_to_stack_layout} from the previous section.
These properties of the following lemma are illustrated in \cref{fig:partition_properties} and then combined in \cref{fig:partition}.

\begin{lemma}
    \label{lem:construct_H_partition}
    Let~$G$ be a maximal outerplanar DAG with fixed base edge and~$T$ be its (rooted) construction tree.
    Then~$G$ admits a directed \Hpartition~$\calP$ with the following properties:
    \begin{enumerate}[(P1)]
        \item\label{prop:parts_maximal_transitive_subgraphs}
        $\calP$ contains exactly one part $P$ for each monotone vertex~$u$ of~$G$ and $P$ contains exactly the vertices of the transitive subgraph\footnote{See \cref{sec:preliminaries} to recall the definition.} below~$u$.
        \item\label{prop:two_paths}
        For each $P \in \calP$ the graph $G[P]$ contains two directed paths $Q_1,Q_2$ such that every vertex of $G - P$ that is stacked onto an edge in $G[P]$ is stacked onto an edge of $Q_1$ or $Q_2$.
        \item\label{prop:H_block_monotone}
        $H$ is a block-monotone outerplanar DAG.
        \item\label{prop:block_cut_cover_number_bounded}
        For each block~$B$ of~$H$ the directed $B$-partition~$\calP_B$ of $G[B]$ has cut cover number at most~$4$.
    \end{enumerate}
\end{lemma}

\begin{figure}
    \centering
    \includegraphics[scale=1.30909090909, page=1]{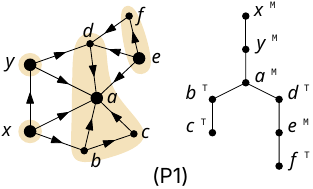}
    \hspace{2cm}
    \includegraphics[scale=1.30909090909, page=2]{partition_properties}

    \vspace{\baselineskip}

    \includegraphics[scale=1.30909090909, page=3]{partition_properties}
    \hspace{2cm}
    \includegraphics[scale=1.30909090909, page=4]{partition_properties}
    \caption{%
        The properties guaranteed by \cref{lem:construct_H_partition}. 
        \Cref{fig:partition} shows a larger example combing all properties.
        \propref{prop:parts_maximal_transitive_subgraphs} 
        Each monotone vertex (thick) has its own part (orange) including all transitive vertices until the next monotone vertex, i.e., the transitive subgraph below it (right). 
        %
        \propref{prop:two_paths}
        Two paths $ Q_1^+,Q_2^+ $ cover all outer edges of part $ P $.
        Adding another transitive vertex $ z $ extends the path (dashed). 
        Removing $ v $ and $ w $ yields $ Q_1 $ and $ Q_2 $.
        \propref{prop:H_block_monotone}
        Each block of the quotient (right) consists of a base edge (thick) and monotone vertices corresponding to the monotone vertices in $ G $ (left).
        \propref{prop:block_cut_cover_number_bounded}
        Inside each block, the cut cover number is bounded, i.e., for each part $ P $ there are four vertices (circled red) that cover all edges leaving $ P $.
    }
    \label{fig:partition_properties}
\end{figure}

\begin{proof}
    We divide the proof into five parts.
    First we define~$\calP$ according to property~\propref{prop:parts_maximal_transitive_subgraphs}.
    Second, we analyze the subgraph of $G$ induced by the vertices of each part $P \in \calP$ and thereby verify property~\propref{prop:two_paths}.
    Third, we show that $\calP$ is indeed a directed \Hpartition, i.e., the quotient~$H := G/\calP$ is a well-defined directed graph.
    Then we show~\propref{prop:H_block_monotone}, i.e., that~$H$ is block-monotone.
    At last, we prove that the cut cover number for each block of $H$ is at most~$4$, verifying property~\propref{prop:block_cut_cover_number_bounded}.

    \begin{subproof}[Construction of~$\boldsymbol{\calP}$]
    For each monotone vertex $u$ of $G$, let $P(u)$ be the set of vertices in the transitive subgraph below $u$.
    By definition, we have $u \in P(u)$ and that $u$ is the only monotone vertex in $P(u)$.
    Moreover recall that the root of the construction tree $T$ is a monotone vertex and thus every transitive vertex of $G$ lies in $P(u)$ for some monotone $u$.
    Hence $\calP = \{ P(u) \mid u \text{ monotone vertex of $G$}\}$ is indeed a partition of $V(G)$ satisfying~\propref{prop:parts_maximal_transitive_subgraphs}.
	\end{subproof}

    For the remainder of the proof it will be convenient to consider a construction sequence of $G$ in which every monotone vertex $u$ is immediately followed by the vertices in the transitive subgraph below $u$.
    We consider such a sequence vertex by vertex and argue about intermediate versions of $G$ and $\calP$ (and thus of $H = G/\calP$).
    At the beginning we have only the base edge $xy$ directed from $x$ to $y$, which are both labeled \lab{M}.
    Then $\calP = \{P(x), P(y)\}$ with $P(x) = \{x\}$ and $P(y) = \{y\}$.
    In each subsequent step, a vertex $u$ is stacked onto an edge $vw$, where $ v,w $ are the parents of $ u $.
    Recall, however, that the \emph{parent of $u$ in the construction tree $T$} is the younger among $v,w$.\footnote{We remark that, unless stated otherwise, parent/child refers to the construction sequence.}
    If~$u$ is a transitive vertex (i.e., labeled~\lab{T} in~$T$), then~$u$ is simply added to the part $P$ in $\calP$ that contains the parent of $u$ in $T$.
    Otherwise, if~$u$ is a monotone vertex (i.e., labeled~\lab{M} in~$T$), then a new part $P(u)$ is added to~$\calP$ consisting of only~$u$.
    \Cref{fig:partition,fig:partition_properties} show an example of the resulting partition.
    Note that this iterative process indeed eventually results in the partition $\calP = \{P(u) \mid u \text{ monotone vertex of $G$}\}$ of $G$ as described above.

    \begin{figure}
        \centering
        \includegraphics[scale=1.30909090909]{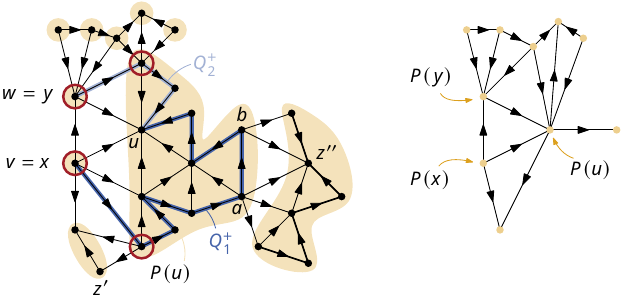}
        \caption{%
            Full example for the proof of \cref{lem:construct_H_partition}.
            For better readability, each step is shown separately in \cref{fig:partition_properties}.
            Left: An outerplanar DAG with base edge $ xy $ and a directed \Hpartition $\calP$ (orange). The vertex $ u $ is stacked onto the edge $ vw $ and is the unique monotone vertex in $ P(u) $. The paths $ Q_1^+ $ (darkblue from $ v $ to $ u $) and $ Q_2^+ $ (lightblue from $ w $ to $ u $) for $ P(u) $ are drawn thick. The vertices marked with red circles certify that the cut cover number within the large block (all parts except for the rightmost) of the part $ P(u) $ is at most 4. In a later step, $ z'' $ is stacked onto $ ab $, introducing a bridge that is a new block in  $H $. 
            Right: The quotient $ H = G/\calP $, where each part of $ \calP $ is contracted to a single vertex. The vertices resulting from the parts $ P(x), P(y), $ and $ P(u) $ are labeled.%
        }
        \label{fig:partition}
    \end{figure}
    
    \begin{subproof}[$\boldsymbol{\calP}$ fulfills~\propref{prop:two_paths}{}]
    We shall argue that each part $P \in \calP$ fulfills~\propref{prop:two_paths} by showing this for the moment when $P$ is created in the construction sequence and maintaining~\propref{prop:two_paths} for $P$ whenever $P$ is augmented with a new vertex thereafter.
    So fix a monotone vertex $u$ and consider the step in the construction sequence when the part $P = P(u)$ is created.
    As $P$ is created with just the single vertex $u$, property~\propref{prop:two_paths} holds vacuously immediately after its creation.
    Moreover, if $u = x$, i.e., the tail of the base edge $ xy $, then $P = P(u)$ will never be augmented with a new vertex and \propref{prop:two_paths} holds throughout.
    To show that~\propref{prop:two_paths} is maintained for $P = P(u)$ for $u \neq x$ with each step of the construction sequence, we maintain two directed paths $Q_1^+,Q_2^+$ associated to $P$ and containing all edges of the subgraph induced by $ P $ onto which a vertex can be stacked (see again \cref{fig:partition,fig:partition_properties}).
    If $u = y$, we initialize both $Q^+_1 $ and $Q^+_2 $ with $ (x, y) $.
    If $u \neq x,y$ and the monotone vertex $u$ is stacked onto edge~$vw$, we initialize $Q_1^+ := (v,u)$ and $Q_2^+ := (w,u)$.
    In any case, we have a $u$-$v$-path $Q_1^+$ and a $u$-$w$-path $Q_2^+$ (taking $ v = w = x $ if $ u = y $), where in $G$ both paths are either consistently oriented towards $u$ or both consistently oriented away from $u$.
    It holds that the next vertex in the construction sequence with a parent in $P$ is stacked onto an edge of $Q_1^+$ or an edge of $Q_2^+$.
    Moreover, the subpaths $Q_1 := Q_1^+ - \{v\}$ and~$Q_2 := Q_2^+ - \{w\}$ are contained in~$G[P]$ and trivially fulfill~\propref{prop:two_paths} for $P = P(u)$.
    
    Now consider the next step with a transitive vertex~$z$ added to $P = P(u)$.
    (If the next vertex is monotone, then $P$ is final and we are done.)
    Then $z$ is stacked onto an edge~$ab$ of~$Q_1^+$ or~$Q_2^+$ (gray/dashed in \cref{fig:partition_properties}).
    We replace edge~$ab$ in~$Q_1^+$ (or~$Q_2^+$) by the path~$(a,z,b)$.
    This way, $Q_1^+$ and $Q_2^+$ are still oriented in $G$ consistently towards $u$ or away from $u$.
    Since~$G$ is outerplanar, no future vertex is also stacked onto $ab$.
    Hence, the next vertex with a parent in $P$ is again stacked onto an edge of $Q_1^+$ or an edge of $Q_2^+$.
    It follows that the subpaths $Q_1 := Q_1^+ - \{v\}$ and~$Q_2 := Q_2^+ - \{w\}$ are contained in~$G[P]$ and again fulfill~\propref{prop:two_paths} for $P = P(u)$.
	\end{subproof}

   \begin{subproof}[$\boldsymbol{\calP}$ is a directed $\boldsymbol{H}$-partition]
    We show that in each step of the construction sequence~$\calP$ is indeed a directed \Hpartition, i.e., all edges between any two parts $P(u_1),P(u_2) \in \calP$ are oriented in the same direction.
    Assume without loss of generality that monotone vertex $u_1$ appears before monotone vertex $u_2$ in the construction sequence.
    Thus at the time $u_2$ is stacked onto some edge, $P(u_1)$ is already in $\calP$ and 
    contains at least one parent of $ u_2 $.
    If $u_2$ is a right (left) child, we show that all edges are oriented from $P(u_1)$ to $P(u_2)$, respectively from $P(u_2)$ to $P(u_1)$.
    This clearly holds immediately after the construction step for $u_2$.

    By symmetry, assume that $u_2$ is a right child.
    Then the paths $Q_1^+$ and $Q_2^+$ for $P(u_2)$ are both in $G$ directed towards $u_2$ and away from the parents of $u_2$.
    Now consider the next step with a transitive vertex $z$ added to $P(u_2)$ such that $z$ also has one of its parents in $P(u_1)$.
    Then $z$ is stacked onto the first edge of $Q_1^+$ or the first edge of $Q_2^+$.
    As $z$ is transitive, the edge between $P(u_1)$ and $z$ is oriented towards $z \in P(u_2)$, as desired.
    \end{subproof}

    \begin{subproof}[$\boldsymbol{H}$ is a block-monotone DAG]
    To show that~\propref{prop:H_block_monotone} holds, we start by observing that~$H = G/\calP$ is a DAG.
    By the previous paragraph, $H$ is a well-defined directed graph, but we need to argue that it is acyclic\footnote{
        The quotient of an outerplanar DAG obtained by contracting some edges might be cyclic.
        For example consider a $6$-cycle with alternating edge orientations.
        Contracting a maximum matching results in a directed $3$-cycle.
    }.
    Since all edges between two parts of~$\calP$ are oriented consistently, it is sufficient to verify that~$H$ remains acyclic whenever a new part $P$ is created.
    So we have $P = P(u)$ for some monotone vertex $u$ and either both edges incident to~$u$ are oriented towards~$u$ or both edges are oriented away from~$u$.
    Thus,~$v_P$ is not on any cycle in~$H$.

    Further observe that each part in $\calP$ induces a connected subgraph of $G$.
    Thus, the quotient graph $H = G/\calP$ is a minor of $G$.
    In particular, $H$ is outerplanar and every block of $H$ is maximal outerplanar.
    We refer to \cref{fig:partition_properties,fig:partition} for examples of quotients.

    It remains to show for~\propref{prop:H_block_monotone} that~$H$ is block-monotone.
    Again we do this using the construction sequence of~$G$.
    Initially, there is only the base edge~$xy$, and we have $\calP = \{P(x),P(y)\}$ with $P(x) = \{x\}$ and $P(y) = \{y\}$.
    We declare the currently only edge of~$H$ to be the base edge for the currently only block of~$H$.

    Now assume again that vertex~$u$ is stacked onto an edge~$vw$ in $G$.
    If~$u$ is a transitive vertex, nothing needs to be done as~$H$ does not change.
    If~$u$ is a monotone vertex and~$v,w$ are in different parts in $\calP$, then some block $B$ of~$H$ is extended by a new vertex $v_P$ for $P = P(u)$.
    Since the stacking is monotone, the enlarged block $B$ of~$H$ remains monotone with respect to the same base edge.
    Lastly, if~$u$ is a monotone vertex and~$v,w$ are in the same part $P' \in \calP$, then~$H$ gets extended by a bridge~$e$ between $v_{P'}$ and the new vertex $v_P$ for $P=P(u)$.
    This bridge~$e$ forms a new block of~$H$, which is monotone with base edge $e$.
    \end{subproof}

	\begin{subproof}[Each block has cut cover number at most 4]
    Let~$B$ be a block of~$H$ and~$G[B] \subseteq G$ be the corresponding subgraph of~$G$.
    Further, let~$P$ be a part of~$\calP_B$, i.e., such that $v_P$ lies in $B$.
    The goal is to show that part $P$ has cut cover number at most~$4$ within the block $B$, i.e., to find a set $S$ of at most four vertices in $G[B]$ that cover all edges of~$G[B]$ with one endpoint in~$P$ and the other endpoint in~$B$ but in another part of~$\calP_B$.
    Recall that the endpoints of every edge in $G$ are in a parent/child relation by construction of the outerplanar graph, where each child has two parents independent of the edge directions.
    In particular, we shall consider edges $e \in E(G[B])$ whose parent-endpoint lies in $P$ while the child-endpoint does not, and edges whose child-endpoint is in $P$ while the parent-endpoint is not.

    \begin{figure}
        \begin{subfigure}{0.4\linewidth}
            \centering
            \includegraphics[scale=1.30909090909, page=1]{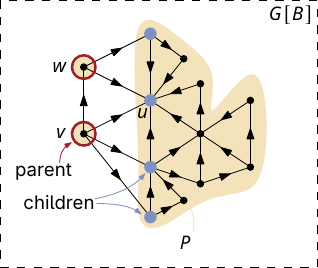}
            \caption{All edges with the child-endpoint in $ P $ but the parent-endpoint not are incident to $ v $ or $ w $.\\~\\~\\~}
            \label{fig:partition_cut-cover-number_case-one}
        \end{subfigure}
        \hfill
        \begin{subfigure}{0.55\linewidth}
            \centering
            \includegraphics[scale=1.30909090909, page=2]{partition_cut-cover-number}
            \caption{All edges with the parent-endpoint in $ P $ but the child-endpoint not are incident to $ q_1 $ or $ q_2 $. Note that stacking onto any edge with both endpoints in $ P $ yields a new block and thus is not considered here, while stacking onto an edge with both endpoints not in $ P $ creates only edges that do not intersect $ P $.}
            \label{fig:partition_cut-cover-number_case-two}
        \end{subfigure}

        \vspace{\baselineskip}

        \begin{subfigure}{\linewidth}
            \centering
            \includegraphics[scale=1.30909090909, page=3]{partition_cut-cover-number}
            \caption{Left: If $G[B]$ does not contain $ v, w $, we choose $ S = \{a,b\}$. Again, recall that stacking onto any edge in part $ P $ opens a new block. Right: The quotient with block $ B $}
            \label{fig:partition_cut-cover-number_case-three}
        \end{subfigure}
        \caption{Bounding the cut cover number in the proof of \cref{lem:construct_H_partition}. Each subfigure refines \cref{fig:partition} to show the details in the different cases of the proof. Recall that the cut cover number is bounded for each block separately.}
        \label{fig:partition_cut-cover-number}
    \end{figure}
        
    Let~$u$ be the monotone vertex of $G$ with $P = P(u)$.
    Clearly, if $u=x$, then $P(u) = \{u\}$ and it is enough to take $S = \{u\}$.
    If $u \neq x$, let $Q_1^+$ and $Q_2^+$ be the paths associated to $P(u)$ as constructed above.
    By symmetry assume that $Q_1^+$ is a directed $v$-to-$u$-path and $Q_2^+$ is a directed $w$-to-$u$-path in $G$ (either because $u=y$ or $u \neq y$ is some right child).
    Let us first assume that $G[B]$ contains $v$ and $w$.
    In this case let $S$ be the set consisting of $v$, $w$, the neighbor $ q_1 $ of $v$ in $Q_1^+$, and the neighbor $ q_2 $ of $w$ in $Q_2^+$, see \cref{fig:partition,fig:partition_cut-cover-number}.
    (Several of these vertices might coincide.)
    Then $|S| \leq 4$ and we claim that $S$ covers every edge in $G[B]$ with exactly one endpoint in $P$.
    Indeed, every edge in $G[B]$ with child-endpoint in $P$ but parent-endpoint outside $ P $ has as parent-endpoint either $v$ or $w$ (\cref{fig:partition_cut-cover-number_case-one}).
    So let $e$ be an edge with parent-endpoint $ p $ in $P$ but child-endpoint $z'$ in $G[B] - P$, see \cref{fig:partition_cut-cover-number_case-two}.
    Hence the part $P'$ in $\calP_B$ containing $z'$ was created later in the construction sequence than $P$.
    To show that $ p $ is either $ q_1 $ or $ q_2 $ and therefore in $ S $, we make two observations:
    First, recall that every monotone vertex that is stacked onto an edge with both endpoints in $ P $ creates a new block.
    And second, observe that immediately after finishing $ P $ in the construction sequence, the only two edges with exactly one endpoint in $ P $ are $ v q_1 $ and $ w q_2 $, i.e., the first edge in each of $ Q_1^+ $ and $ Q_2^+ $.
    Together, it follows that whenever a new monotone vertex is introduced in the construction sequence, either it has both parents in $ P $ and creates a new block, or it has exactly one parent in $ P $, namely $ q_1 $ or $ q_2 $, or it has no parent in $ P $.
    Since we only consider the block $ B $, we conclude that all vertices introduced after finishing $ P $ that have at least one parent in $ P $ actually have exactly one parent in $ P $, namely $ q_1 $ or $ q_2 $.

    Finally, assume that $G[B]$ does not contain $v$ and $w$ (e.g., \cref{fig:partition_cut-cover-number_case-three} or if $ B $ is the bridge in \cref{fig:partition}).
    Then $v_P$ is a cut vertex of $H$ and $B$ is a child block of $v_P$ in the block-cut tree of $ H $. 
    In particular, $ P $ is the first part in $ G[B] $ according to the construction sequence, and thus every edge in $G[B]$ with exactly one endpoint in $P$ has the parent-endpoint in $P$ and the child-endpoint in $G[B]-P$.
    Recall that block $B$ was initialized as a bridge when a monotone vertex $z''$ was stacked onto an edge $ab$ of $Q_1$ or $Q_2$.
    In this case let $S = \{a,b\}$.
    Then $|S| = 2$ and $S$ clearly contains the parent-endpoint of every edge in $G[B]$ with exactly one vertex in~$P$.
    Since we only need to bound the cut cover number of the partition restricted to some block $B$, this finishes~\propref{prop:block_cut_cover_number_bounded}.
    \end{subproof}

    We conclude that \calP is indeed a directed $H$-partition satisfying all four properties, which concludes the proof.
\end{proof}

\subsection{A Constant Upper Bound}

By Property \propref{prop:parts_maximal_transitive_subgraphs} of \cref{lem:construct_H_partition}, each part of the constructed directed \Hpartition induces the transitive subgraph below a monotone vertex.
Bounding the stack number of these subgraphs is the last missing piece before we prove \cref{thm:outerplanar_bounded}.

\begin{lemma}
    \label{lem:maximal_transitive_subgraph}
    Let $ G $ be a maximal outerplanar DAG with a fixed base edge and let $ T $ be its (rooted) construction tree.
    Then the transitive subgraph below every monotone vertex $ w $ admits a 1-stack layout.
\end{lemma}

\begin{sidefigure}
    \centering
    \includegraphics[scale=1.30909090909]{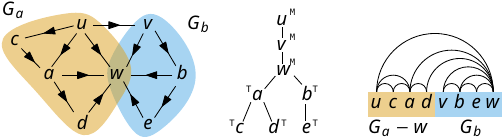}
    \caption{An outerplanar DAG with its construction tree and a 1-stack layout of the transitive subgraph below $ w $}
    \label{fig:transitive_1-stack}
\end{sidefigure}

\begin{proof}
    Let $ uv $ denote the parent edge of $ w $.
    By symmetry, we assume that $ w $ is a right child.
    First recall that $ w $ has at most two children $ a $ and $ b $ in $ T $, corresponding to the children of $ uw $, respectively $ vw $, in $ G $.
    Let $ G_a $ ($ G_b$) denote the subgraph of $ G $ induced by $ u $ ($ v $), $ w $, and the transitive subgraph below $ a $ ($ b $), see \cref{fig:transitive_1-stack}.
    Note that the union of $ G_a $ and $ G_b $ contains the transitive subgraph below $ w $.
    Now observe that $ G_a $ and $ G_b $ are transitive outerplanar DAGs, i.e., obtained from $ uw $, respectively $ vw $, by repeatedly stacking transitive children.
    Therefore, they have a unique topological ordering which coincides with the ordering of the vertices around the outer face, starting with the tail of their base edge and ending with its head.
    It is well known that one stack suffices for outerplanar graphs with this vertex ordering.
    A 1-stack layout for $ G_a \cup G_b $ is now obtained by concatenating the layouts of $ G_a - w $ and $ G_b $, which in particular gives a 1-stack layout for the transitive subgraph below~$ w $.
\end{proof}

\medskip

Finally, we are ready to prove the first main result of this paper.
This includes verifying the premises of \cref{lem:H_partition_to_stack_layout}, which we restate here for convenience.

\HPartitionToStackLayout*

\outerplanarBounded*

\begin{proof}
    Without loss of generality we assume that~$G$ is a maximal outerplanar DAG.
    This is justified as the stack number is monotone under taking subgraphs and because a non-maximal outerplanar DAG~$G$ can easily be extended into a maximal one:
    Add (undirected) edges to~$G$ as long as the underlying undirected graph remains outerplanar.
    Then take any topological ordering~$\prec$ of~$G$ and orient each added edge from its left endpoint in~$\prec$ to its right endpoint in~$\prec$.

    Fix a base edge for $G$ and hence the corresponding construction tree.
    By \cref{lem:construct_H_partition} there is a directed \Hpartition~$\calP$ of~$G$ satisfying properties \propref{prop:parts_maximal_transitive_subgraphs}--\propref{prop:block_cut_cover_number_bounded}.
    In particular, by property~\propref{prop:H_block_monotone} the DAG~$H$ is block-monotone.
    Thus by \cref{cor:block_monotone_bounded} there is an $h$-stack layout~$\prec_H$ of~$H$ with~$h \leq 258$.
    Further, by property~\propref{prop:parts_maximal_transitive_subgraphs}, every part of~$\calP$ induces a transitive subgraph below some monotone vertex of~$G$ and as such admits a $1$-stack layout by \cref{lem:maximal_transitive_subgraph}.

    Now we seek to apply \cref{lem:H_partition_to_stack_layout}, for which we have to check its premises:
    \begin{itemize}
        \item By property~\propref{prop:block_cut_cover_number_bounded}, for every block $B$ of $H$ the directed $B$-partition $\calP_B$ of $G[B]$ has cut cover number at most~$w = 4$.
        Thus, by \cref{lem:bounded_H_partition_to_stack_layout}, graph~$G[B]$ admits an $H$-expanding stack layout using at most $s \leq 3wh + 1 = 3 \cdot 4 \cdot 258 + 1 = 3097$ stacks.

        \item By property~\propref{prop:two_paths}, in each part $P \in \calP$ there are $p = 2$ directed paths $Q_1,Q_2$ in $G[P]$ such that all vertices stacked onto edges of $G[P]$ are stacked onto an edge of $Q_1$ or $Q_2$.
        As~$G$ is outerplanar, we additionally get that at most~$t = 1$ vertex is stacked onto each of those edges.
    \end{itemize}
    Seeing all premises fulfilled, \cref{lem:H_partition_to_stack_layout} yields that the stack number of~$G$ is at most $\sn(G) \leq 4spt \leq 4 \cdot 3097 \cdot 2 \cdot 1 = 24776$.
\end{proof}

\section{Directed Acyclic 2-Trees have Unbounded Stack Number}
\label{sec:2trees}

We construct a directed acyclic $2$-tree~$G$ with arbitrarily large twist number (hence arbitrarily large stack number) in every topological vertex ordering~$\prec$.
Somewhat surprisingly, we first consider rainbows, which can be seen as the counterpart to twists and are defined as follows.
A \emph{$k$-rainbow}, $k \geq 1$, is a set of~$k$ edges that pairwise nest with respect to~$\prec$.
While in general, vertex orderings with small stack number (hence small twist number) are allowed to have arbitrarily large rainbows, we first argue that there is a very large rainbow in~$\prec$ for our constructed $2$-tree~$G$, and then use that rainbow as a lever to slowly find larger and larger twists in~$\prec$.

We start with a straightforward auxiliary lemma.
For this consider a triangle with vertices $ u \prec v \prec w $ with vertex ordering $ \prec $.
Then we call $ u $ the \emph{left vertex}, $ v $ the \emph{middle vertex}, and $ w $ the right vertex.
We further call a set of pairwise vertex-disjoint triangles \emph{well-interleaved} with respect to some vertex ordering if we first have all left vertices, then all middle vertices, and finally all right vertices.
Note that the ordering of the vertices within each group is not determined.

For the proof, we use the Erd\H{o}s-Szekeres theorem which states that every sequence of length $ p \cdot q $ consisting of pairwise distinct integers contains a monotonically increasing sequence of length $ p $ or a monotonically decreasing sequence of length $ q $.

\begin{lemma}
    \label{lem:triangles}
    If $ k^3 $ triangles are well-interleaved with respect to a vertex ordering $ \prec $, then there is a $ k $-twist.
\end{lemma}

\begin{proof}
    Consider well-interleaved triangles $ T_i = (u_i, v_i, w_i) $ for $ i = 1, \dots, k^3 $.
    Without loss of generality we have $ u_1 \prec \dots \prec u_{k^3} \prec v_1, \dots, v_{k^3} \prec w_1, \dots, w_{k^3} $, that is, only the ordering within the $ v $-vertices and within the $ w $-vertices is unknown.
    Among the $ v $-vertices, the Erd\H{o}s-Szekeres theorem with $ p = k $ and $ q = k^2 $ yields an increasing sequence of $ k $ indices, i.e., a $ k $-twist between $ u $- and $ v $-vertices with which we are done, or a decreasing sequence of $ k^2 $ indices with which we continue.
    From now on, we only consider these $ k^2 $ indices.
    Again, by Erd\H{o}s-Szekeres, there either is an increasing sequence of length $ k $ among the $ k^2 $ considered $ w $-vertices yielding a $ k $-twist with the corresponding $ u $-vertices.
    Or there is a decreasing sequence of length $ k $ giving a $ k $-twist between the $ v $- and $ w $-vertices.
\end{proof}

\twotreesUnbounded*

We remark that we actually prove a slightly stronger statement, namely that the twist number is at least~$ k $, which in turn is a lower bound on the stack number.

\begin{proof}
    Let~$k \geq 1$ be fixed.
    Below we construct a $2$-tree~$G$ with twist number~$\tn(G) \geq k$.
    The proof is split into two parts.
    First we construct the $2$-tree~$G$ before proving that every topological vertex ordering contains a $k$-twist.

    \begin{subproof}[Construction of~$\boldsymbol{G}$.]
    We define the desired $2$-tree~$G$ via a sequence of $2$-trees $G_0 \subset G_1 \subset \cdots \subset G_k$ with $G_k = G$.
    For each $t = 0, \ldots, k$ we shall have a matching $E_t \subset E(G_t)$ such that in~$G_t$ no vertex is stacked onto any edge in~$E_t$, and $E_0, \ldots, E_k$ are pairwise disjoint.

    We start with~$G_0$ being a single edge~$ab$ oriented from~$a$ to~$b$, and $E_0 = \{ab\}$.
    Having defined~$G_t$ and~$E_t$ for some $0 \leq t < k$, we define~$G_{t+1}$ and~$E_{t+1}$ as follows:
    We use each edge~$ab \in E_t$ (directed from~$a$ to~$b$) as the base edge for a particular $2$-tree that we denote by~$T(ab)$ and that is constructed as follows:
    Let~$N$ be a large enough integer
    (to be specified below).
    \begin{itemize}
        \item Add a sequence $b^1, \ldots, b^N$ of vertices, where~$b^j$ is stacked as a right child onto the edge $ab^{j-1}$ (putting $b^0 = b$).
        \item Add a sequence $a^1, \ldots, a^N$ of vertices, where~$a^j$ is stacked as a left child onto the edge~$ab^j$.
        \item Denote by~$E(ab)$ the matching $E(ab) = \{a^jb^j
        \mid j \in \{1, \ldots, N\}\}$.
    \end{itemize}
    See \cref{fig:2-tree-construction} for an illustration.
    \begin{sidefigure}
        \centering%
        \includegraphics[scale=1.30909090909]{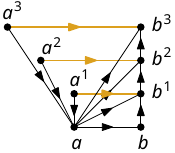}
        \caption{Construction of $ T(ab) $ with the edge set $E(ab)$ (orange)}
        \label{fig:2-tree-construction}
    \end{sidefigure}
    Observe that~$T(ab)$ involves no transitive stackings, that at most two vertices are stacked onto each edge of~$T(ab)$, and that in~$T(ab)$ no vertex is stacked onto an edge in~$E(ab)$.
    Further observe that in every vertex ordering of~$T(ab)$ the vertices $b^1, \ldots, b^N$ come (actually in that order) to the right of~$a$ and~$b$, while the vertices $a^1, \ldots, a^N$ come (not necessarily in that order) to the left of~$a$ and~$b$.
    In particular,~$a$ and~$b$ are consecutive in every vertex ordering of~$T(ab)$.

    Now let~$G_{t+1}$ be the $2$-tree obtained from~$G_t$ by adding the $2$-tree~$T(ab)$ onto all edges $ab \in E_t$.
    Further let~$E_{t+1}$ be the union of the matchings~$E(ab)$ for all edges~$ab \in E_t$.
    Apart from the exact value of~$N$, this completes the definition of~$G_{t+1}$ on the basis of~$G_t$, and hence the definition of $G = G_k$.

    For our proof below to work, we require that~$N$ is an enormous (but constant) number in terms of~$k$.
    We set~$N = r_1k$, where~$r_1, \ldots, r_k$ is a sequence of integers defined recursively by
    \begin{align*}
        r_k &= 1 \qquad \text{and} \\
        r_t &= 2 \cdot k^3 (2 k^3)^{1 + r_{t+1} k} \qquad \text{for $t = k-1, \ldots, 1$}
        \text{.}\qedhere 
    \end{align*}
	\end{subproof}

	\begin{subproof}[Twist number of~$\boldsymbol{G}$.]
    Let $ \prec $ be an arbitrary vertex ordering of $ G $.
    We give an inductive proof of a slightly stronger statement than the existence of a $ k $-twist, for which we need the following definition.
    For positive integers $ r, t $ we call a matching a \emph{$ t $-twist with an $ r $-thick edge} if it is obtained from a $ t $-twist by replacing one edge by an $ r $-rainbow.
    In particular, the $ r $-rainbow is vertex-disjoint from the remaining $ (t-1) $-twist and each rainbow-edge crosses each twist-edge.
    Our long-term goal is to find such a $ t $-twist with a thick edge in $ G_t $. 
    However, we may also find a $ k $-twist along the way, in which case we can stop immediately.
    Thus, throughout the proof we always assume that we do not find a $ k $-twist.
    Under this assumption, we now give our stronger statement that we prove by induction:
    For every $ t = 1, \dots, k $, the subgraph $ G_t $ of $ G $ contains a $ t $-twist with an $ r_t $-thick edge, where the $ r_t $-rainbow consists of edges in $E_t$.
    We start with a huge rainbow, which rapidly decreases while increasing the size of the twist by 1 in each step until we obtain a $ k $-twist for $ t = k $.
    \end{subproof}
    
    For $ t = 1 $, a 1-twist with an $ r_1 $-thick edge is simply an $ r_1 $-rainbow.
    Recall that~$G_1 = T(ab)$ for the only edge~$ab \in E_0 = E(G_0)$.
    As mentioned above, we have $a^1, \ldots, a^N \prec a \prec b \prec b^1 \prec b^2 \prec \cdots \prec b^N$.
    By the Erd\H{o}s-Szekeres theorem, the ordering of~$N$ such $a$-vertices according to~$\prec$ contains a $k$-element subsequence with monotonically increasing indices or an $N/k$-element subsequence with monotonically decreasing indices.
    The former case gives a $k$-twist from the corresponding $a$-vertices to the $b$-vertices, so we can stop.
    Otherwise, the latter case gives a rainbow formed by $N/k = r_1$ edges of $ E_1 $, as desired.
    
    Now, for $ t \geq 1 $, assume that we have a $ t $-twist $ T $ with an $ r_t $-thick edge, where $ R \subseteq E_t $ denotes the $ r_t $-rainbow.
    We aim to find an entirely new rainbow $ R' \subseteq E_{t+1} $ of size $ r_{t+1} $ and an edge $ e' $ crossing all edges of $ R' $, where all these edges start in the region spanned by the starting points of $ R $ and end in the region spanned by the endpoints of $ R $.
    Together with the $ t - 1 $ edges of $ T - R $, this forms a $ (t + 1) $-twist with an $ r_{t+1} $-thick edge, see \cref{fig:thick_edge}.
    \begin{figure}
        \centering
        \includegraphics[scale=1.30909090909]{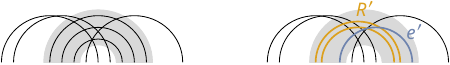}
        \caption{%
            Left: A 4-twist with a 3-thick edge. 
            Right: A 2-rainbow $R'$ crossed by an edge $ e'$ resulting in a 5-twist with a 2-thick edge.
        }
        \label{fig:thick_edge}
    \end{figure}
    For this, recall that in the construction of $ G $, the 2-tree $ T(ab) $ is added to each edge $ ab \in R \subseteq E_t $.
    To find $ R' $ and $ e' $, we follow two steps: 
    First, we analyze the edges $ ab \in R $ and their right children. 
    Here, we either find a large subset of edges having their children far to the right (\cref{fig:right_children}~left), which yields well-interleaved triangles and thus a $ k $-twist by \cref{lem:triangles}.
    Or we have the other extreme that many edges of $ R $ have their children close to their right endpoint (\cref{fig:right_children}~right).
    In the second step, we find $ R' $ and $ e' $ in the 2-tree $ T(ab) $ of such an edge~$ab$.
    
    \begin{cfigure}
        \centering
        \includegraphics[scale=1.30909090909, page=1]{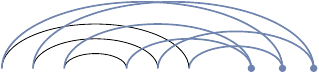}
        \hspace{4em}
        \includegraphics[scale=1.30909090909, page=3]{right_children}
        \caption{%
            Left: Three consecutive edges of the rainbow $ R $ with their first right children (blue vertices) skipping sufficiently many vertices of $ R $ such that three well-interleaved triangles are formed.
            (Note that for simplification of the proof, \emph{all} children are required to skip at least $ k^3 $ vertices of $ R $, although this is not necessary and not shown here for better readability, e.g., the child of the outermost edge.)
            Right: Two edges of the form $ ab \in S_s \subseteq R $ with the right children of $ T(ab)$ immediately following~$ b $.
        }
        \label{fig:right_children}
    \end{cfigure}
    
    \begin{cfigure}
        \centering
        \includegraphics[scale=1.30909090909, page=2]{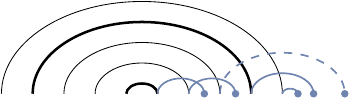}
        \caption{%
            Five consecutive edges of $ R $ with their first right children, edges to the left parent are omitted for readability.
            Edges to children skipping more than one vertex of $ R $ are drawn dashed. 
            Each block of two consecutive edges of $ R $ contains an edge whose child skips at most one vertex. 
            The chosen subrainbow $ S_1 $ is highlighted thick.
        }
        \label{fig:find_good_right_children}
    \end{cfigure}
    
    For the first part, consider an edge $ ab \in R $ and the first right child $ b^1 $ in $ T(ab) $, and count the number of vertices of $ R $ that are \emph{skipped}, i.e., the number of endpoints of edges of $ R $ between $ b $ and $ b^1 $.
    If there are $ k^3 $ (along the nesting order) consecutive edges $ ab $ in $ R $ such that each $ b^1 $ skips at least $ k^3 $ vertices of $ R $, then we obtain $ k^3 $ well-interleaved triangles (\cref{fig:right_children}~left) and therefore a $ k $-twist by \cref{lem:triangles}.
    As we are done in this case, we assume the other case in which we have a set $ S \subseteq R $ with the following two properties:
    First, $ S $ consists of at least $ r_t / k^3 $ edges whose first right child skips less than $ k^3 $ vertices of $ R $.
    And second, among each $ k^3 $ consecutive edges of $ R $ at least one belongs to $ S $.
    It follows that among the edges of $ S $, there is a subset $ S_1 \subseteq S $ of $ r_t / (2k^3) $ edges $ ab $ such that the right child $ b^1 $ does not skip any vertex of $ S_1 $; choose one in each block of $2k^3$ consecutive edges in $ R $ (\cref{fig:find_good_right_children}).
    
    Next, consider the second right children, i.e., the vertex $ b^2 $ for edges $ ab \in S_1 $.
    With the same argument as above, we have one of the two extremes in \cref{fig:right_children}:
    Either $ k^3 $ consecutive edges of $ S_2 $ have their second right children far to the right forming well-interleaved triangles.
    Or there is a subrainbow $ S_2 \subseteq S_1 \subseteq S \subseteq R $ of size  $ r_t / (2k^3)^2 $ such that for every edge $ ab \in S_2 $, its second right child $ b^2 $ does not skip any vertex of $ S_2 $.
    Repeating the argument $ s $ times, we obtain a subrainbow of $ S_s \subseteq \dots \subseteq S_2 \subseteq S_1 \subseteq R $ such that each edge $ ab $ has many right children $ b^1 , b^2, \dots, b^s $ not skipping any vertex of $ S_s $.
    Note that the size of the subrainbow shrinks by a factor of $ 2 k^3 $ in each step.
    Since $ r_t = k^3 (2k^3)^s $ for $ s = 1 + r_{t+1} k $, after $ s $ steps we are left with a set $ S_s \subseteq R $ of $ k^3 $ edges of the form $ ab $ such that $ b^1 \prec \dots \prec b^s $ immediately follow $ b $, i.e., no other vertex of $ S_s $ or the considered right children is between $ b $ and $ b^s $, which concludes the first part.
    
    \begin{figure}
        \centering
        \includegraphics[scale=1.30909090909]{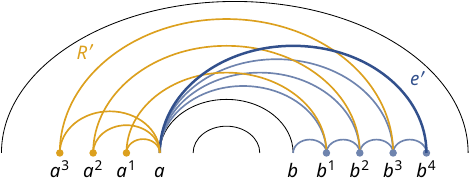}
        \caption{The new rainbow $ R' $ and the edge $ e' = a b^s $ crossing $ R' $}
        \label{fig:new_rainbow}
    \end{figure}
    
    The second part considers the left children of the 2-trees $ T(ab) $ to find $ R' $ and $ e' $.
    Observe that if all $ k^3 $ edges of $ S_s \subseteq R $ have a left child outside the region spanned by $ R $, then these children together with their parents form $ k^3 $ well-interleaved triangles.
    As this yields a $ k $-twist by \cref{lem:triangles}, we may assume that there is an edge $ ab \in S_s $ with the children $ a^1, \dots, a^{s-1} $ below the outermost edge of $ R $, where again $ s = 1 + r_{t+1} k $.
    Among these children, we find the starting points of the new rainbow $ R' $ using the Erd\H{o}s-Szekeres theorem:
    Either we find a sequence of $ k $ increasing indices, then the respective $ a^i b^i $-edges form a $ k $-twist and we are done.
    Or we find a sequence of $ r_{t+1} $ decreasing indices, then the respective $ a^i b^i $-edges form our desired $ r_{t+1} $-rainbow $ R' $.
    Finally, we choose $ e' = a b^s $ as an edge that crosses all edges of $ R' $, see \cref{fig:new_rainbow}.
    Combing the $ (t-1) $-twist $ T - R $ with the edge $ e' $ and the rainbow $ R' \subseteq E_{t+1} $, and recalling that the 2-tree $ T(ab) $ including $ R' $ and $ e'$ is below the outermost edge of $ R $, we obtain a $ (t+1) $-twist with an $ r_{t+1} $-thick edge.
\end{proof}

\section{Conclusion and Open Problems}
We proved that outerplanar DAGs have bounded stack number (\cref{thm:outerplanar_bounded}) and that monotone $2$-trees with at most two vertices stacked on every edge have unbounded stack number (\cref{thm:2tree_unbounded}).
In both cases we solved long-standing open problems or conjectures.
In doing so we got pretty close to pinpointing the boundary between bounded and unbounded stack number of directed acyclic $2$-trees, hence DAGs of treewidth~$2$.
However, several interesting questions worth to be considered remain open.

Our first open problem is about the exact upper bound for the stack number of outerplanar DAGs.
We are certain that the number of~$24776$ stacks required by our approach can be lowered.
The best known lower bound is an outerplanar DAG (that is even upward planar) presented by N\"{o}llenburg and Pupyrev~\cite{Noellenburg2021_DAGsWithConstantStackNumber} that requires four stacks.

\begin{open}
    What is the largest stack number of outerplanar DAGs exactly?
\end{open}

Our family of $2$-trees constructed to prove \cref{thm:2tree_unbounded} is not upward planar.
This motivates the following open problem to further narrow the gap between bounded and unbounded stack number.

\begin{open}
    Is the stack number of upward planar $2$-trees bounded?
\end{open}

In fact, this is just a special case of the same question for general upward planar graphs.
Here the best lower bound is an upward planar graph requiring five stacks~
compared to an $O((n \log n)^{2/3})$ upper bound, where $n$ is the number of vertices~\cite{Jungeblut2022_SublinearUpperBound}.

\begin{open}
    Is the stack number of upward planar graphs bounded?
\end{open}

\printbibliography

@article{Alam2021_DispersableBookEmbeddings,
	author = {Md. Jawaherul Alam and Michael A. Bekos and Vida Dujmović and Martin Gronemann and Michael Kaufmann and Sergey Pupyrev},
	doi = {10.1016/j.tcs.2021.01.035},
	issn = {0304-3975},
	journal = {Theoretical Computer Science},
	pages = {1–22},
	title = {On dispersable book embeddings},
	volume = {861},
	year = {2021}
}

@article{Alam2022_MixedPageNumber,
	author = {Md. Jawaherul Alam and Michael A. Bekos and Martin Gronemann and Michael Kaufmann and Sergey Pupyrev},
	doi = {10.1016/j.tcs.2022.07.036},
	issn = {0304-3975},
	journal = {Theoretical Computer Science},
	pages = {131–141},
	title = {The mixed page number of graphs},
	volume = {931},
	year = {2022}
}

@article{Alhashem2015_StacksPoset,
    author = {Alhashem, Mustafa
        and Jourdan, Guy-Vincent
        and Zaguia, Nejib},
    title = {{On The Book Embedding Of Ordered Sets}},
    journal = {Ars Combinatoria},
    year = {2015},
    volume = {119},
    pages = {47--64}
}

@inproceedings{Alzohairi1997_SeriesParallelPosets,
    author = {Alzohairi, Mohammad
        and Rival, Ivan},
    title = {{Series-Parallel Planar Ordered Sets Have Pagenumber Two}},
    booktitle = {Graph Drawing (GD 1996)},
    year = {1997},
    editor = {North, Stephen},
    volume = {1190},
    series = {Lecture Notes in Computer Science},
    pages = {11--24},
    doi = {10.1007/3-540-62495-3_34}
}

@article{Alzohairi2001_Lattices,
    author = {Alzohairi, Mohammad
        and Rival, Ivan
        and Kostochka, Alexandr},
    title = {{The Pagenumber of Spherical Lattices is Unbounded}},
    journal = {Arab Journal of Mathematical Sciences},
    year = {2001},
    volume = {7},
    number = {1},
    pages = {79--82}
}

@article{Angelini2022_MixedLayouts2Trees,
    author = {Angelini, Patrizio 
        and Bekos, Michael A.
        and Kindermann, Philipp
        and Mchedlidze, Tamara},
    title = {On mixed linear layouts of series-parallel graphs},
    journal = {Theoretical Computer Science},
    volume = {936},
    pages = {129--138},
    year = {2022},
    issn = {0304-3975},
    doi = {10.1016/j.tcs.2022.09.019},
}

@article{Bannister2019_TrackLayouts,
    author = {Bannister, Michael J.
        and Devanny, William E.
        and Dujmovi\'{c}, Vida
        and Eppstein, David
        and Wood, David R.},
    title = {Track Layouts, Layered Path Decompositions, and Leveled Planarity},
    journal = {Algorithmica},
    year = {2019},
    volume = {81},
    pages = {1561--1583},
    doi = {10.1007/s00453-018-0487-5}
}

@article{Bekos2016_Deg4Hamiltonian,
    author = {Bekos, Michael A.
        and Gronemann, Martin
        and Raftopoulou, Chrysanthi N.},
    title = {{Two-Page Book Embeddings of 4-Planar Graphs}},
    journal = {Algorithmica},
    year = {2016},
    volume = {75},
    pages = {158--185},
    doi = {10.1007/s00453-015-0016-8}
}

@inproceedings{Bekos2019_Planar_kPlanar,
    author = {Bekos, Michael A.
        and Da Lozzo, Giordano
        and Dujmovi{\'{c}}, Vida
        and Frati, Fabrizio
        and Gronemann, Martin
        and Mchedlidze, Tamara
        and Montecchiani, Fabrizio
        and N{\"{o}}llenburg, Martin
        and Pupyrev, Sergey
        and Raftopoulou, Chrysanthi N.},
    title = {{On Linear Layouts of Planar and k-Planar Graphs}},
    booktitle = {Beyond-Planar Graphs: Combinatorics, Models and Algorithms (Dagstuhl Seminar 19092)},
    year = {2019},
    editor = {Hong, Seok-Hee
        and Kaufmann, Michael
        and Pach, J{\'{a}}nos
        and T{\'{o}}th, Csaba D.},
    volume = {9(2)},
    series = {Dagstuhl Reports},
    pages = {144--148},
    doi = {10.4230/DagRep.9.2.123}
}

@article{Bekos2020_4StacksPlanar,
    author = {Bekos, Michael A.
        and Kaufmann, Michael
        and Klute, Fabian
        and Pupyrev, Sergey
        and Raftopoulou, Chrysanthi N.
        and Ueckerdt, Torsten},
    title = {{Four Pages are Indeed Necessary for Planar Graphs}},
    journal = {Journal of Computational Geometry},
    year = {2020},
    volume = {11},
    number = {1},
    pages = {332--353},
    doi = {10.20382/jocg.v11i1a12}
}

@article{Bekos2022_DAGs2StacksNP,
    title = {{Recognizing DAGs with page-number 2 is {NP}-complete}},
    journal = {Theoretical Computer Science},
    volume = {946},
    pages = {113689},
    year = {2023},
    issn = {0304-3975},
    doi = {https://doi.org/10.1016/j.tcs.2023.113689},
    author = {Michael A. Bekos and Giordano {Da Lozzo} and Fabrizio Frati and Martin Gronemann and Tamara Mchedlidze and Chrysanthi N. Raftopoulou},
}

@article{Bernhart1979_BookThickness,
    author = {Bernhart, Frank
        and Kainen, Paul C.},
    title = {{The Book Thickness of a Graph}},
    journal = {Journal of Combinatorial Theory, Series B},
    year = {1979},
    volume = {27},
    number = {3},
    pages = {320--331},
    doi = {10.1016/0095-8956(79)90021-2}
}

@article{Bhore2021_UpwardBookThickness,
	author = {Sujoy Bhore and Giordano {Da Lozzo} and Fabrizio Montecchiani and Martin Nöllenburg},
	doi = {10.1016/j.ejc.2022.103662},
	issn = {0195-6698},
	journal = {European Journal of Combinatorics},
	pages = {103662},
	title = {On the upward book thickness problem: Combinatorial and complexity results},
	volume = {110},
	year = {2023}
}

@article{Binucci2019_UpwardStackST,
  author       = {Carla Binucci and
                  Giordano Da Lozzo and
                  Emilio Di Giacomo and
                  Walter Didimo and
                  Tamara Mchedlidze and
                  Maurizio Patrignani},
  title        = {Upward Book Embeddability of st-Graphs: Complexity and Algorithms},
  journal      = {Algorithmica},
  volume       = {85},
  number       = {12},
  pages        = {3521--3571},
  year         = {2023},
  doi          = {10.1007/S00453-023-01142-Y},
}

@inproceedings{Buss1984_StacksPlanar,
    author = {Buss, Jonathan F.
        and Shor, Peter W.},
    title = {{On the Pagenumber of Planar Graphs}},
    booktitle = {Proceedings of the Sixteenth Annual ACM Symposium on Theory of Computing (STOC 1984)},
    year = {1984},
    pages = {98--100},
    doi = {10.1145/800057.808670}
}

@article{Davies2022_ColoringCircleGraphs,
    author = {Davies, James},
    title = {{Improved bounds for colouring circle graphs}},
    journal = {Proceedings of the American Mathematical Society},
    year = {2022},
    doi = {10.1090/proc/16044},
    pages = {5121-5135},
    volume={150},
    number={12}
}

@inproceedings{deCol2019_MixedLinearLayours,
	address = "Cham",
	author = "de Col, Philipp and Klute, Fabian and Nöllenburg, Martin",
	booktitle = "Graph Drawing and Network Visualization (GD 2019)",
	editor = "Archambault, Daniel and Tóth, Csaba D.",
	isbn = "978-3-030-35802-0",
	pages = "460–467",
	publisher = "Springer International Publishing",
	title = "Mixed Linear Layouts: Complexity, Heuristics, and Experiments",
	year = "2019",
    doi = "10.1007/978-3-030-35802-0_35"
}

@article{DeFraysseix1995_PlanarBipartite,
    author = {de Fraysseix, Hubert
        and de Mendez, Patrice O.
        and Pach, J{\'{a}}nos},
    title = {{A Left-First Search Algorithm for Planar Graphs}},
    journal = {Discrete \& Computational Geometry},
    year = {1995},
    volume = {13},
    number = {3--4},
    pages = {459--468},
    doi = {10.1007/BF02574056}
}

@article{DiGiacomo2006_SeriesParallel,
    author = {Di Giacomo, Emilio
        and Didimo, Walter
        and Liotta, Giuseppe
        and Wismath, Stephen K.},
    title = {{Book Embeddability of Series–Parallel Digraphs}},
    journal = {Algorithmica},
    year = {2006},
    volume = {45},
    pages = {531--547},
    doi = {10.1007/s00453-005-1185-7}
}

@article{Dujmovic2004_LinearLayouts,
    title = {On Linear Layouts of Graphs},
    author = {Dujmovi{\'{c}}, Vida
        and Wood, David R.},
    url = {https://dmtcs.episciences.org/317},
    doi = {10.46298/dmtcs.317},
    journal = {{Discrete Mathematics \& Theoretical Computer Science}},
    volume = {{Vol. 6 no. 2}},
    year = {2004},
    month = Jan,
}

@article{Dujmovic2005_TrackLayouts,
    author = {Dujmovi{\'{c}}, Vida
        and Morin, Pat
        and Wood, David R.},
    title = {{Layout of Graphs with Bounded Tree-Width}},
    journal = {SIAM Journal on Computing},
    year = {2005},
    volume = {34},
    number = {3},
    pages = {553--579},
    doi = {10.1137/S0097539702416141}
}

@article{Dujmovic2007_treewidth,
  author    = {Vida Dujmovic and
               David R. Wood},
  title     = {Graph Treewidth and Geometric Thickness Parameters},
  journal   = {Discret. Comput. Geom.},
  volume    = {37},
  number    = {4},
  pages     = {641--670},
  year      = {2007},
  doi       = {10.1007/s00454-007-1318-7},
}

@article{Dujmovic2020_PlanarGraphsBoundedQueueNumber,
    author = {Dujmovi{\'{c}}, Vida
        and Joret, Gwena{\"{e}}l
        and Micek, Piotr
        and Morin, Pat
        and Ueckerdt, Torsten
        and Wood, David R.},
    title = {{Planar Graphs Have Bounded Queue-Number}},
    journal = {Journal of the ACM},
    year = {2020},
    volume = {67},
    number = {4},
    pages = {1--38},
    doi = {10.1145/3385731}
}

@inproceedings{Foerster2023_BipartiteLinearLayouts,
    doi = "10.1007/978-3-031-38906-1_29",
	address = "Cham",
	author = "Förster, Henry and Kaufmann, Michael and Merker, Laura and Pupyrev, Sergey and Raftopoulou, Chrysanthi",
	booktitle = "Algorithms and Data Structures (WADS 2023)",
	editor = "Morin, Pat and Suri, Subhash",
	isbn = "978-3-031-38906-1",
	pages = "444–459",
	publisher = "Springer Nature Switzerland",
	title = "Linear Layouts of Bipartite Planar Graphs",
	year = "2023"
}

@article{Frati2013_UpwardStackNumber,
    author = {Frati, Fabrizio
        and Fulek, Radoslav
        and Ruiz-Vargas, Andres J.},
    title ={{On the Page Number of Upward Planar Directed Acyclic Graphs}},
    journal = {Journal of Graph Algorithms and Applications},
    year = {2013},
    volume = {17},
    number = {3},
    pages = {221--244},
    doi = {10.7155/jgaa.00292}
}

@article{Ganley2001_Treewidth,
    author = {Ganley, Joseph L.
        and Heath, Lenwood S.},
    title = {{The pagenumber of $k$-trees is $O(k)$}},
    journal = {Discrete Applied Mathematics},
    year = {2001},
    volume = {109},
    number = {3},
    pages = {215--221},
    doi = {10.1016/S0166-218X(00)00178-5}
}

@article{Guan2020_Deg5Planar,
    author = {Guan, Xiaxia
        and Yang, Weihua},
    title = {{Embedding planar 5-graphs in three pages}},
    journal = {Discrete Applied Mathematics},
    year = {2020},
    volume = {282},
    pages = {108--121},
    doi = {10.1016/j.dam.2019.11.020}
}

@article{Gyarfas1985_CircleGraphsChiBounded,
    author = {Gy{\'{a}}rf{\'{a}}s, Andr{\'{a}}s},
    title = {{On the Chromatic Number of Multiple Interval Graphs and Overlap Graphs}},
    journal = {Discrete Mathematics},
    year = {1985},
    volume = {55},
    number = {2},
    pages = {161--166},
    doi = {10.1016/0012-365X(85)90044-5}
}

@article{Hakimi1996_StarAboricity,
    author = {Hakimi, Seifollah L.
        and Mitchem, John
        and Schmeichel, Edward F.},
    title = {{Star arboricity of graphs}},
    journal = {Discrete Mathematics},
    year = {1996},
    volume = {149},
    number = {1--3},
    pages = {93--98},
    doi = {10.1016/0012-365X(94)00313-8}
}

@inproceedings{Heath1984_7StacksPlanar,
    author = {Heath, Lenwood S.},
    title = {{Embedding Planar Graphs In Seven Pages}},
    booktitle = {25th Annual Symposium on Foundations of Computer Science (FOCS 1984)},
    year = {1984},
    pages = {74--83},
    doi = {10.1109/SFCS.1984.715903}
}

@article{Heath1992_Comparing,
    author = {Heath, Lenwood S.
        and Leighton, Frank Thomson
        and Rosenberg, Arnold L.},
    title = {{Comparing Queues and Stacks As Machines for Laying Out Graphs}},
    journal = {SIAM Journal on Discrete Mathematics},
    year = {1992},
    volume = {5},
    number = {3},
    pages = {398--412},
    doi = {10.1137/0405031}
}

@article{Heath1992_Queues,
    author = {Heath, Lenwood S.
        and Rosenberg, Arnold L.},
    title = {{Laying Out Graphs Using Queues}},
    journal = {SIAM Journal on Computing},
    year = {1992},
    volume = {21},
    number = {5},
    pages = {927--958},
    doi = {10.1137/0221055}
}

@article{Heath1997_Posets,
    author = {Heath, Lenwood S.
        and Pemmaraju, Sriram V.},
    title = {{Stack and Queue Layouts of Posets}},
    journal = {SIAM Journal on Discrete Mathematics},
    year = {1997},
    volume = {10},
    number = {4},
    pages = {599--625},
    doi = {10.1137/S0895480193252380}
}

@article{Heath1999_DAGs1,
    author = {Heath, Lenwood S.
        and Pemmaraju, Sriram V.
        and Trenk, Ann N.},
    title = {{Stack and Queue Layouts of Directed Acyclic Graphs: Part I}},
    journal = {SIAM Journal on Computing},
    year = {1999},
    volume = {28},
    number = {4},
    pages = {1510--1539},
    doi = {10.1137/S0097539795280287}
}

@article{Heath1999_DAGs2,
    author = {Heath, Lenwood S.
        and Pemmaraju, Sriram V.},
    title = {{Stack and Queue Layouts of Directed Acyclic Graphs: Part II}},
    journal = {SIAM Journal on Computing},
    year = {1999},
    volume = {28},
    number = {5},
    pages = {1588--1626},
    doi = {10.1137/S0097539795291550}
}

@article{Hung1993_Poset4Stacks,
    author = {Hung, Le Tu Quoc},
    title = {{A Planar Poset which Requires 4 Pages}},
    journal = {Ars Combinatoria},
    year = {1993},
    volume = {35},
    pages = {291--302}
}

@article{Jungeblut2022_SublinearUpperBound,
    author       = {Paul Jungeblut and
                    Laura Merker and
                    Torsten Ueckerdt},
    title        = {A Sublinear Bound on the Page Number of Upward Planar Graphs},
    journal      = {{SIAM} J. Discret. Math.},
    volume       = {37},
    number       = {4},
    pages        = {2312--2331},
    year         = {2023},
    doi          = {10.1137/22M1522450},
    bibsource    = {dblp computer science bibliography, https://dblp.org}
    }

@INPROCEEDINGS{Jungeblut2023_FOCSversion,
  author={Jungeblut, Paul and Merker, Laura and Ueckerdt, Torsten},
  booktitle={2023 IEEE 64th Annual Symposium on Foundations of Computer Science (FOCS)}, 
  title={Directed Acyclic Outerplanar Graphs Have Constant Stack Number}, 
  year={2023},
  pages={1937-1952},
  doi={10.1109/FOCS57990.2023.00118}
}

@inproceedings{Kainen1974_BookThickness,
	address = "Berlin, Heidelberg",
	author = "Kainen, Paul C.",
	booktitle = "Graphs and Combinatorics",
	editor = "Bari, Ruth A. and Harary, Frank",
	isbn = "978-3-540-37809-9",
	pages = "76–108",
	publisher = "Springer Berlin Heidelberg",
	title = "Some recent results in topological graph theory",
	year = "1974"
}

@inproceedings{Knauer2012_SimpleTreewidth,
    author = {Knauer, Kolja
        and Ueckerdt, Torsten},
    title = {{Simple Treewidth}},
    booktitle = {Midsummer Combinatorial Workshop Prague},
    year = {2012},
    editor = {Ryt{\'{i}}{\v{r}}, Pavel},
    pages = {21--23},
    url = {https://kam.mff.cuni.cz/workshops/work18/mcw2012booklet.pdf}
}

@article{Malitz1994_EulerGenus,
    author = {Malitz, Seth M.},
    title = {{Genus $g$ Graphs Have Pagenumber $O(\sqrt{g})$}},
    journal = {Journal of Algorithms},
    year = {1994},
    volume = {17},
    number = {1},
    pages = {85--109},
    doi = {10.1006/jagm.1994.1028}
}

@inproceedings{Mchedlidze2009_HamiltonianCompletion,
    author = {Mchedlidze, Tamara
        and Symvonis, Antonios},
    title = {{Crossing-Free Acyclic Hamiltonian Path Completion for Planar st-Digraphs}},
    booktitle = {Algorithms and Computation (ISAAC 2009)},
    year = {2009},
    editor = {Dong, Yingfei
        and Du, Ding-Zhu
        and Ibarra, Oscar},
    volume = {5878},
    series = {Lecture Notes in Computer Science},
    pages = {882--891},
    doi = {10.1007/978-3-642-10631-6_89}
}

@mastersthesis{Merker2020_Thesis,
  author = {Merker, Laura},
  title = {{Ordered Covering Numbers}},
  school = {Karlsruhe Institute of Technology},
  year = {2020},
  url = {https://i11www.iti.kit.edu/_media/teaching/theses/ma-merker-20.pdf}
}

@inproceedings{Merker2019_LocalStack,
    author = {Merker, Laura
        and Ueckerdt, Tortsen},
    title = {{Local and Union Page Numbers}},
    booktitle = {Graph Drawing and Network Visualization (GD 2019)},
    year = {2019},
    editor = {Archambault, Daniel
        and T{\'{o}}th, Csaba D.},
    volume = {11904},
    series = {Lecture Notes in Computer Science},
    pages = {447--459},
    doi = {10.1007/978-3-030-35802-0_34}
}

@inproceedings{Merker2020_LocalQueue,
    author = {Merker, Laura
        and Ueckerdt, Torsten},
    title = {{The Local Queue Number of Graphs with Bounded Treewidth}},
    booktitle = {Graph Drawing and Network Visualization (GD 2020)},
    year = {2020},
    editor = {Auber, David
        and Valtr, Pavel},
    volume = {12590},
    series = {Lecture Notes in Computer Science},
    pages = {26--39},
    doi = {10.1007/978-3-030-68766-3_3}
}

@inproceedings{Merker2021_LocalComplete,
    author = {Felsner, Stefan
        and Merker, Laura
        and Ueckerdt, Torsten
        and Valtr, Pavel},
    title = {{Linear Layouts of Complete Graphs}},
    booktitle = {Graph Drawing and Network Visualization (GD 2021)},
    year = {2021},
    editor = {Purchase, Helen C.
        and Rutter, Ignaz},
    volume = {12868},
    series = {Lecture Notes in Computer Science},
    pages = {257--270},
    doi = {10.1007/978-3-030-92931-2_19}
}

@inproceedings{Noellenburg2021_DAGsWithConstantStackNumber,
	address = "Cham",
	author = "Nöllenburg, Martin and Pupyrev, Sergey",
	booktitle = "Graph Drawing and Network Visualization (GD 2023)",
	editor = "Bekos, Michael A. and Chimani, Markus",
	isbn = "978-3-031-49272-3",
	pages = "135–151",
	publisher = "Springer Nature Switzerland",
	title = {{On Families of Planar DAGs with Constant Stack Number}},
	year = "2023",
    doi = "10.1007/978-3-031-49272-3_10"
}

@article{Nowakowski1989_DirectedStackNumer,
    author = {Nowakowski, Richard
        and Parker, Andrew},
    title = {{Ordered Sets, Pagenumbers and Planarity}},
    journal = {Order},
    year = {1989},
    volume = {6},
    pages = {209--218},
    doi = {10.1007/BF00563521}
}

@inproceedings{Ollmann1973_BookThickness,
    author = {Ollmann, L. Taylor},
    title = {{On the Book Thicknesses of Various Graphs}},
    booktitle = {Proc. 4th Southeastern Conference on Combinatorics, Graph Theory and Computing},
    year = {1973},
    volume = {8}
}

@InProceedings{Pupyrev2017_MixedLayoutPlanar,
    author = {Pupyrev, Sergey},
    editor = {Frati, Fabrizio
        and Ma, Kwan-Liu},
    title = {Mixed Linear Layouts of Planar Graphs},
    booktitle = {Graph Drawing and Network Visualization (GD 2017)},
    year = {2018},
    publisher = {Springer International Publishing},
    address = {Cham},
    pages = {197--209},
    isbn = {978-3-319-73915-1},
    doi = {10.1007/978-3-319-73915-1_17}
}

@article{Pupyrev2020_TrackNumber,
   author = {Pupyrev, Sergey},
   title = {Improved Bounds for Track Numbers of Planar Graphs},
   journal = {Journal of Graph Algorithms and Applications},
   year = {2020},
   volume = {24},
   number = {3},
   pages = {323--341},
   doi = {10.7155/jgaa.00536}
}

@inproceedings{Rengarajan1995_Stacks2Trees,
    author = {Rengarajan, S.
        and Veni Madhavan, C. E.},
    title = {{Stack and Queue Number of 2-Trees}},
    booktitle = {Computing and Combinatorics (COCOON 1995)},
    year = {1995},
    editor = {Du, Ding-Zhu
        and Li, Ming},
    volume = {959},
    series = {Lecture Notes in Computer Science},
    pages = {203--212},
    doi = {10.1007/BFb0030834}
}

@inproceedings{Syslo1990_Posets,
      author = {Sys{\l}o, Maciej M.},
    title = {{Bounds to the Page Number of Partially Ordered Sets}},
    booktitle = {Graph-Theoretic Concepts in Computer Science (WG 1989)},
    year = {1990},
    editor = {Nagl, Manfred},
    volume = {411},
    series = {Lecture Notes in Computer Science},
    pages = {181--195},
    doi = {10.1007/3-540-52292-1_13}
}

@article{Togasaki2002_Pathwidth,
    author = {Togasaki, Mitsunori
        and Yamazaki, Koichi},
    title = {{Pagenumber of pathwidth-$k$ graphs and strong pathwidth-$k$ graphs}},
    journal = {Discrete Mathematics},
    year = {2002},
    volume = {259},
    number = {1--3},
    pages = {361--368},
    doi = {10.1016/S0012-365X(02)00542-3}
}

@article{Tutte1956_4ConnectedHamiltonian,
    author = {Tutte, William T.},
    title = {{A Theorem on Planar Graphs}},
    journal = {Transactions of the American Mathematical Society},
    year = {1956},
    volume = {82},
    number = {1},
    pages = {99--116},
    doi = {10.1090/S0002-9947-1956-0081471-8}
}

@article{Vandenbussche2009_kTrees,
    author = {Vandenbussche, Jennifer and West, Douglas B. and Yu, Gexin},
    title = {On the Pagenumber of k-Trees},
    journal = {SIAM Journal on Discrete Mathematics},
    volume = {23},
    number = {3},
    pages = {1455-1464},
    year = {2009},
    doi = {10.1137/080714208},
}

@techreport{Wigderson1982_HamiltonianCompletion,
    author = {Wigderson, Avi},
    title = {{The Complexity of the Hamiltonian Circuit Problem for Maximal Planar Graphs}},
    institution = {Electrical Engineering and Computer Science Department, Princeton University},
    year = {1982},
    url = {https://www.math.ias.edu/~avi/PUBLICATIONS/MYPAPERS/W82a/tech298.pdf}
}

@mastersthesis{Wulf2016_SimpleTreewidth,
    author = {Wulf, Lasse},
    title = {{Stacked Treewidth and the Colin de Verdi\`{e}re Number}},
    school = {Karlsruhe Institute of Technology},
    year = {2016},
    type = {Bachelor's thesis},
    url = {https://i11www.iti.kit.edu/_media/teaching/theses/ba-wulf-16.pdf}
}

@article{Yannakakis1989_4StacksPlanar,
    author = {Yannakakis, Mihalis},
    title = {{Embedding Planar Graphs in Four Pages}},
    journal = {Journal of Computer and System Sciences},
    year = {1989},
    volume = {38},
    number = {1},
    pages = {36--67},
    doi = {10.1016/0022-0000(89)90032-9}
}

@article{Yannakakis2020_4StacksPlanar,
    author = {Yannakakis, Mihalis},
    title = {{Planar Graphs that Need Four Pages}},
    journal = {Journal of Combinatorial Theory, Series B},
    year = {2020},
    volume = {145},
    pages = {241--263},
    doi = {10.1016/j.jctb.2020.05.008}
}

@InProceedings{Auer2011_QueueDequeGraphs,
    author="Auer, Christopher
    and Bachmaier, Christian
    and Brandenburg, Franz Josef
    and Brunner, Wolfgang
    and Glei{\ss}ner, Andreas",
    editor="Brandes, Ulrik
    and Cornelsen, Sabine",
    title="Plane Drawings of Queue and Deque Graphs",
    booktitle="Graph Drawing (GD 2010)",
    year="2011",
    publisher="Springer Berlin Heidelberg",
    address="Berlin, Heidelberg",
    pages="68--79",
    doi={10.1007/978-3-642-18469-7_7}
}

@misc{Bekos2023_ComputeLinearLayouts,
      title={An Online Framework to Interact and Efficiently Compute Linear Layouts of Graphs}, 
      author={Michael A. Bekos and Mirco Haug and Michael Kaufmann and Julia Männecke},
      year={2023},
      eprint={2003.09642},
      archivePrefix={arXiv},
      primaryClass={cs.DM},
      note={arXiv:2003.09642},
      doi={10.48550/arXiv.2003.09642}
}

@inproceedings{Bekos2023_DequeRiqueComplete,
    author       = {Michael A. Bekos and
                    Michael Kaufmann and
                    Maria Eleni Pavlidi and
                    Xenia Rieger},
    editor       = {Denis Pankratov},
    title        = {On the Deque and Rique Numbers of Complete and Complete Bipartite
                    Graphs},
    booktitle    = {Proceedings of the 35th Canadian Conference on Computational Geometry,
                    {CCCG} 2023, Concordia University, Montreal, Quebec, Canada, July
                    31 - August 4, 2023},
    pages        = {89--95},
    year         = {2023},
    bibsource    = {dblp computer science bibliography, https://dblp.org},
    url          = {https://cccg.ca/proceedings/2023/CCCG2023.pdf}
    }

@InProceedings{Bekos2023_RiqueNumber,
    author="Bekos, Michael A.
    and Felsner, Stefan
    and Kindermann, Philipp
    and Kobourov, Stephen
    and Kratochv{\'{i}}l, Jan
    and Rutter, Ignaz",
    editor="Angelini, Patrizio
    and von Hanxleden, Reinhard",
    title="The Rique-Number of Graphs",
    booktitle="Graph Drawing and Network Visualization (GD 2022)",
    year="2023",
    publisher="Springer International Publishing",
    address="Cham",
    pages="371--386",
    doi={10.1007/978-3-031-22203-0_27}
}
\end{document}